\documentclass{amsart}
\usepackage{amsmath, amsthm, amssymb}
\usepackage{amsmath, amssymb}
 \usepackage{amsmath,amscd,amsthm}
\usepackage{mathrsfs}
\usepackage{tikz}
\input xypic
\xyoption{all}

\theoremstyle{plain} 
\newtheorem{theorem}[subsection]{Theorem}
\newtheorem{proposition}[subsection]{Proposition}
\newtheorem{lemma}[subsection]{Lemma}

\theoremstyle{definition}

\theoremstyle{remark}

\newtheorem*{remark*}{Remark}

\numberwithin{equation}{subsection}

\begin{document}

\title[Order Types of Reflection Orders on affine Weyl groups]{Order Types of Reflection Orders on affine Weyl groups}

\author{Weijia Wang}
\address{School of Statistics and Mathematics \\Shanghai Lixin University of Accounting And Finance\\ No. 995 Shangchuan Road, Shanghai, 201209\\ China }
\email{wwang8@alumni.nd.edu}
\author{Rui Wang}
\address{School of Statistics and Mathematics \\Shanghai Lixin University of Accounting And Finance\\ No. 995 Shangchuan Road, Shanghai, 201209\\ China }
\email{18709483691@163.com}

\begin{abstract}
In this paper, we investigate the order types of reflection orders on irreducible affine Weyl groups. We show that they are intimately related to the Catalan combinatorics.
We explicitly describe all of those order types and show that these order types can be parameterized by deformed Dyck words, which can be enumerated by Catalan numbers.
\end{abstract}
\maketitle

\section{Introduction}

Reflection orders, introduced by Dyer \cite{dyerhecke}, are specific total orderings defined on the set of positive roots (or equivalently, the set of reflections) of a Coxeter group. Since then, reflection orders have found interesting applications in Bruhat orders, Iwahori-Hecke algebra, Kazhdan-Lusztig polynomials and R-polynomials. For example, one can employ reflection orders to derive (nonrecursive) combinatorial formulae for Kazhdan-Lusztig polynomials and R-polynomials. Despite these wide applications, reflection orders are not well understood and many  questions  concerning their structure remain unanswered. For example, there are two  long standing conjectures of Dyer concerning reflection orders. The first conjecture asserts that a maximal chain of biclosed sets in the set of positive roots consists precisely of the initial intervals of a reflection order. The second conjecture asserts that the initial intervals of all reflection orders on a Coxeter group when ordered under inclusion form a lattice.

A fundamental combinatorial question in the study of a total order is the determination of its order type.
In this study, we initiate an investigation of the possible order types of reflection orders on infinite Coxeter groups. In particular we present explicit description of the order types of irreducible affine Weyl groups.  Cellini and Papi \cite{cellini} proved that any reflection order on affine Weyl groups can be expressed as a shuffle of a
finite sum of totally ordered sets of type $\omega$ and a finite sum of totally ordered sets of
type $\omega^*$. We present a finer description and the primary result of the study is that, for affine Weyl groups, the order types of reflection orders can be parameterized by ``deformed" Dyck words.

Here, an \emph{extended Dyck word} of length $2n$ is defined to be a string consisting of $n$ occurrences of 1 and $n$ occurrences of 0, such that every proper initial segment of it has more 0s than 1s. Take an extended Dyck word of length $2n$ and substitute some of its contiguous substrings of the form $\underbrace{\epsilon\epsilon\cdots\epsilon}_{k\, \mathrm{times}}$ ($\epsilon\in \{0,1\}, k\geq 2$) with $\underbrace{\epsilon\epsilon\cdots\epsilon}_{k'\, \mathrm{times}}, 0<k'<k$.  We call the resulting word a $2n$-\emph{trimmed Dyck word}. Given a $2n-$trimmed Dyck word $\epsilon_1\epsilon_2\cdots\epsilon_k,$ one can define an associated order type:
$o(\epsilon_1)+o(\epsilon_2)+\cdots +o(\epsilon_k)$
where $o(0)=\omega$, the order type of the natural order on the set of natural numbers and $o(1)=\omega^*$, the order type of the backward order of the natural order on the set of natural numbers. Our results assert that after removing finitely many roots, each order type of the reflection order on an affine Weyl group of rank $n$ is associated to a $2n$-trimmed Dyck word and conversely for every $2n$-trimmed Dyck word there exists a reflection order whose order type is associated to it. See Theorem \ref{mainthm} for the precise statement and Example \ref{keyexample} for a concrete example. The analysis in this study will be applied to the investigation of braid operation on reflection orders in the subsequent works. Furthermore, we show that those $2n-$trimmed Dyck words are almost Catalan objects, in the sense that the number of $2n-$trimmed Dyck words equals to the $(n+1)$-th Catalan number minus one.

This paper is organized as follow: In Section \ref{bkg} we review the basic properties of linear ordering and present background materials from Coxeter theory. Section \ref{condensation} is dedicated to the description of a specific method for partitioning a reflection order on an affine Weyl group into a finite set of intervals. This partition (which is called a condensation) plays a pivotal role in our analysis.
In Section \ref{mainthm}, we demonstrate that the condensation described in Section \ref{condensation} yields a chain in the poset of biclosed sets in the associated finite root system. The signature of such a chain is related to the deformed Dyck words, which allows us to prove the main theorem on enumeration.
In the appendix, we prove that $2n-$trimmed Dyck words can be counted by using Catalan numbers.

\section{Background}\label{bkg}

\subsection{Total orders}

In this subsection we review necessary background materials from the theory of linear (total) ordering. A good reference is \cite{linearorder} and we usually follow the terminology and notation used there.
Two totally ordered sets $(A,\prec_A), (B,\prec_B)$ are said to be isomorphic if there exists a bijection $f:A\rightarrow B$ such that
$$f(a_1)\prec_B f(a_2)\,\,\,\text{if}\,\text{and}\,\text{only}\,\text{if}\,\,\, a_1\prec_A a_2.$$
This notion of isomorphism  defines an equivalence relation on the class of all totally ordered sets.
Two equivalent totally ordered sets are said to be of the same order type.
A totally ordered set  is said to have the \emph{order type} $\theta$ if $\theta$ is the representative of the equivalence class of this totally ordered set.
Denote the set of natural number by $\mathbb{N}$.
The order type of $\mathbb{N}$ (with the natural order) is denoted by $\omega$. All  totally ordered sets with $n$ elements are isomorphic. Their order type is denoted by $[n]$.
The backward total order of a totally ordered set $(A,\prec_A)$ (denoted by $(A,\prec_A^*)$) is defined by $x\prec_A^* y$ if and only if $y\prec_A x$.
If the order type of $(A,\prec_A)$ is $\theta$, the order type of $(A,\prec_A^*)$ will be denoted by $\theta^*.$

Let $(A,\prec_A)$ and $(B,\prec_B)$ be two disjoint totally ordered set. Their sum (denoted by $(A,\prec_A)+(B,\prec_B)$) is the totally ordered set
$(A\cup B, \prec_{A\cup B})$ where $x\prec_{A\cup B}y$ if and only if  ($x\in A$ and $y\in B$) or ($x\prec_Ay, x, y\in A$) or $(x\prec_By, x,y\in B)$.
This addition is associative and therefore parentheses will not be written when we add totally ordered sets.
If the order type of $(A,\prec_A)$ is $\theta_1$ and the order type of $(B,\prec_B)$ is $\theta_2$, then the order type of their sum is denoted by $\theta_1+\theta_2.$

We define an \emph{interval} of $(A,\prec_A)$ to be a subset $B$ of $A$ such that  $x,y\in B, z\in A$ with $x\prec z\prec y$ implies $z\in B$.
An \emph{initial interval} (resp. \emph{final interval}) of a totally ordered set $(A,\prec_A)$ is a subset $I\subset A$ such that for any $y\in I$ and $x\prec y$ (resp. $y\prec x$) one has $x\in I.$  Clearly initial intervals and final intervals are both intervals. For $x,y\in A, x\preccurlyeq y$ the subset $\{z\in A\mid\,x\preccurlyeq z\preccurlyeq y\}$ is an interval and will be denoted by $[x,y]$. Let $A$ and $B$ be two initial intervals of a totally ordered set. Clearly one has either $A\subset B$ or $B\subset A.$

Let $(A,\prec)$ be a totally ordered set and let $B\subset A$ be a subset. The total order of $\prec$ restricted on the subset $B$ will be denoted by $\prec\mid_B.$

We list a few lemmas which will be needed in the subsequent analysis. For the sake of brevity, we omit their straightforward proofs.

\begin{lemma}\label{nanddualordertype}
\begin{enumerate}
\item Let $P$ be an infinite totally ordered set. If every initial interval of the form $\{q\in P\mid\,q\preccurlyeq p\}, p\in P$ (resp. final interval $\{q\in P\mid\,p\preccurlyeq q\}, p\in P$) of $P$ is finite, then $P$ has the order type of $\omega$ (resp. $\omega^*$).
\item Let $(A,\prec)$ be a countably infinite totally ordered set. Suppose that $A=B\uplus C$ where $C$ is finite and that $(B,\prec\mid_B)$ is of order type $\omega^*$. Then
$(A, \prec)$ is of order type $[k]+\omega^*$ for some $k\geq 0$.
\item Let $(A,\prec)$ be a totally ordered set. Assume that $Q_i,1\leq i\leq t$ are $t$ nonempty initial intervals of $A$ such that $Q_1\varsubsetneqq Q_2 \varsubsetneqq \cdots \varsubsetneqq Q_t=A$. Write $P_1=Q_1, P_i=Q_i\backslash Q_{i-1}, 2\leq i\leq t$. Then $(A,\prec)$ is isomorphic to
$$(P_1,\prec)+(P_2,\prec)+\cdots+(P_t,\prec).$$
\end{enumerate}
\end{lemma}

One can totally order $\{P_i\}$ in (3) naturally, i.e. $P_i\preccurlyeq' P_j$ if and only if $i\leq j.$ The resulting totally ordered set $(\{P_i\}_{i=1}^{t},\prec')$ (or by abuse of terminology $\{Q_i\}_{i=1}^{t}$) is called a \emph{condensation} of $(A,\prec)$.

\subsection{Root system and reflection order}
We refer the reader to \cite{bjornerbrenti}, \cite{Bourbaki} and \cite{Hum} for the basics of Coxeter groups and their root systems.
Let $(W,S)$ be a Coxeter system.
Its  \emph{root system}, the set of \emph{positive roots}, the set of \emph{negative roots} and the set of \emph{simple roots} are denoted by $\Phi_W, \Phi^+_W$, $\Phi^-_W$ and $\Delta_W$ respectively.
The set of \emph{reflections} of $W$ is the set $\{wsw^{-1}\mid\, w\in W, s\in S\}$. Under the canonical bijection between the set of reflections and $\Phi^+_W$, we denote by $s_{\alpha}$ the  reflection corresponding to the positive root $\alpha\in \Phi^+_W$.

A subset $B$ of $\Phi_W$ is said to be \emph{closed} if for any $\alpha,\beta\in B, k_1,k_2\in \mathbb{R}_{\geq0}$ with $k_1\alpha+k_2\beta\in \Phi$ one has $k_1\alpha+k_2\alpha\in B.$
Let $\Gamma$ be a subset of $\Phi_W.$ A subset $B$ of $\Gamma$ is said to be \emph{biclosed} in $\Gamma$ if both $B$ and $\Gamma\backslash B$ are closed.

There is a complete description of biclosed set in $\Phi_W$ where $\Phi_W$ is an irreducible crystallographic root system of a finite Weyl group $W$ in \cite{biclosedphi} and \cite{DyerReflOrder}.  Denote by $\Phi_{\Delta'}$ the root subsystem of $\Phi$ generated by $\Delta'\subset \Phi$.  The biclosed sets in $\Phi_W$ are  subsets of $\Phi_W$ of the form $(\Psi^+\backslash \Phi_{\Pi_1})\cup \Phi_{\Pi_2}$ where $\Psi^+$ is a positive system of $\Phi_W$ and $\Pi_1,\Pi_2$ are two orthogonal subsets (i.e. $(\alpha,\beta)=0$ for any $\alpha\in \Pi_1,\beta\in \Pi_2$) of the set of simple roots of $\Psi^+.$ Following \cite{DyerReflOrder}  we denote the set $(\Psi^+\backslash \Phi_{\Pi_1})\cup \Phi_{\Pi_2}$ by $\Psi^+_{\Pi_1,\Pi_2}$.

 A \emph{reflection order} $\prec$ of $(W,S)$ is a total order on $\Phi^+_W$ such that for all $\alpha,\beta\in \Phi^+_W$ with $\alpha\prec\beta$ and $k_1,k_2\in \mathbb{R}_{>0}$ such that $k_1\alpha+k_2\beta\in \Phi^+_W$ one has that $\alpha\prec k_1\alpha+k_2\beta\prec \beta.$ Equivalently a total order on $\Phi^+$ is a reflection order if and only if all of its initial intervals are biclosed in $\Phi_W^+.$

For $W$ an irreducible affine Weyl group, its root system $\Phi_W$ can be constructed from the crystallographic root system $\Phi_{W_0}$ of the corresponding irreducible Weyl group $W_0$. Let $V$ be the real Euclidean space containing $\Phi_{W_0}$ such that $\mathbb{R}\Phi_{W_0}=V$. Construct a real vector space $V'=V\oplus \mathbb{R}\delta$ where $\delta$ is an indeterminant. Extend the inner product $(,)$ on $V$ to $V'$ by requiring $(V',\delta)=0.$ For $\alpha\in \Phi_{W_0}^+$ (the chosen standard positive system of $\Phi_{W_0}$), define $\alpha^E$ to be the set $\{\alpha+n\delta\mid\,n\geq 0\}$. For $\alpha\in \Phi_{W_0}^-$, define $\alpha^E$ to be the set $\{\alpha+n\delta\mid\,n>0\}$. For $B\subset \Phi_{W_0}$, define $B^E=\cup_{\alpha\in B}\alpha^E$. Then $\Phi_{W}^+=(\Phi_{W_0})^E$ and $\Delta_{W}=\Delta_{W_0}\cup\{\delta-\rho\}$ where $\rho$ is the highest root of $\Phi_{W_0}^+$. For $\alpha\in \Phi_{W_0}^+$ define $\alpha_0=\alpha$ and for $\alpha\in \Phi_{W_0}^-$ define $\alpha_0=\alpha+\delta.$ Let $B$ be a biclosed sets in $\Phi^+_W$. It can be easily verified that $B_0:=\{\alpha\in \Phi_{W_0}\mid \alpha^E\cap B\,\text{is}\,\text{infinite}\}$ is biclosed in $\Phi_{W_0}.$ In this paper the rank of $W$ is understood to be the rank of its associated Weyl group $W_0$.

For $W$ an irreducible affine Weyl group, the biclosed sets in $\Phi^{+}_W$ are determined in \cite{DyerReflOrder} and recently independently in \cite{combinbiclosed}.
Let $W_i, i=1,2$ be the reflection subgroup of $W$ generated by the reflections corresponding to the roots in $\Phi_{\Delta_i}^E,i =1,2$. Then any biclosed set in $\Phi^+_W$ is of the form $(\Psi^+_{\Delta_1,\Delta_2})^E\cup (\Phi_u\cap \Phi_{\Delta_1}^E)\backslash (\Phi_v\cap \Phi_{\Delta_2}^E)$ for some $\Psi^+,\Delta_1,\Delta_2$ and $u\in W_1, v\in W_2.$

\section{A condensation of the reflection order}\label{condensation}

Now let $W$ be an irreducible affine Weyl group. Denote by $\Phi_{W_0}$ the crystallographic root system of the associated Weyl group $W_0.$
In this section we will divide the reflection order $\prec$ into several intervals whose order types are either $\omega$ or $\omega^*+[k]$ for some nonnegative integer $k$.

\begin{lemma}\label{dihedral}
For a root $\alpha\in \Phi_{W_0}$, the restriction of a reflection order $\prec$ of $W$ on the set $\{\pm\alpha\}^E$  is either
$$(\alpha)_0\prec (\alpha)_0+\delta \prec (\alpha)_0+2\delta \prec \cdots$$
$$\cdots \prec (-\alpha)_0+2\delta\prec (-\alpha)_0+\delta \prec (-\alpha)_0$$ or
$$(-\alpha)_0\prec (-\alpha)_0+\delta \prec (-\alpha)_0+2\delta \prec \cdots$$
$$\cdots \prec (\alpha)_0+2\delta\prec (\alpha)_0+\delta \prec (\alpha)_0.$$
\end{lemma}

\begin{proof}
One notes that the set $\{\pm \alpha\}^E$ is the set of positive roots for the infinite dihedral reflection subgroup $W'=\langle\{s_{(\alpha)_0},s_{(-\alpha)_0}\}\rangle$  (with the simple reflections  $s_{(\alpha)_0},s_{(-\alpha)_0}$). For an infinite dihedral Coxeter system with simple roots $\beta,\gamma$ there  exists only two reflection orders (Definition 2.1 in \cite{dyerhecke}):
$$\beta\prec s_{\beta}(\gamma)\prec s_{\beta}s_{\gamma}(\beta)\prec \cdots \prec s_{\gamma}s_{\beta}(\gamma)\prec s_{\gamma}(\beta)\prec\gamma,$$
$$\gamma\prec s_{\gamma}(\beta)\prec s_{\gamma}s_{\beta}(\gamma)\prec \cdots \prec s_{\beta}s_{\gamma}(\beta)\prec s_{\beta}(\gamma)\prec\beta.$$
Therefore the assertion follows.
\end{proof}
Now we fix a reflection order $\prec$ on $\Phi^+_{W}$.
We call a root $\alpha\in \Phi_{W_0}$  an \emph{initial root} if  one gets
$$(\alpha)_0\prec (\alpha)_0+\delta \prec (\alpha)_0+2\delta \prec \cdots$$
$$\cdots \prec (-\alpha)_0+2\delta\prec (-\alpha)_0+\delta \prec (-\alpha)_0$$
when we restrict $\prec$ on $\{\pm\alpha\}^E$.

\begin{proposition}
Initial roots form a positive system $\Psi^+_{W_0}$ in $\Phi_{W_0}$.
\end{proposition}

\begin{proof}
Clearly $\Psi^+_{W_0}\uplus -\Psi^+_{W_0}=\Phi_{W_0}.$
Suppose that $\gamma=\beta+\eta, \beta,\eta\in \Psi^+_{W_0}$ and $\gamma\in \Phi_{W_0}.$ Take $m,n\in \mathbb{Z}_{>0}.$ Write $p=m+n.$
Then one has either $\beta+n\delta\prec \gamma+p\delta\prec \eta+m\delta$ or $ \eta+m\delta \prec\gamma+p\delta\prec \beta+n\delta$. Without loss of generality, assume that the former is the case.
Suppose that $\gamma\not\in \Psi^+_{W_0}.$ Then $-\gamma+\delta\prec \gamma+p\delta$. Note that $-\eta+(n+1)\delta=(-\gamma+\delta)+(\beta+n\delta)$.
This implies that $-\eta+(n+1)\delta$ is less than one of $\beta+n\delta$ and $-\gamma+\delta$. Hence $-\eta+(n+1)\delta\prec \gamma+p\delta\prec \eta+m\delta.$ A contradiction to the fact $\eta\in \Psi^+_{W_0}$. So $\gamma\in \Psi^+_{W_0}$. Then the assertion follows from \cite{Bourbaki} Chapitre VI, 1.7, Corollaire 1.
\end{proof}

From now on, we will denote the positive system in the previous proposition by $\Psi^+_{W_0,\prec}.$

Let us describe the way that we construct our condensation before delving into the precise definition.
For a root $\alpha\in \Psi^+_{W_0,\prec}$, we will associate two initial intervals. This first interval $J(\alpha)$ is the minimal initial interval containing all roots in the chain $$(\alpha)_0\prec (\alpha)_0+\delta \prec (\alpha)_0+2\delta \prec \cdots.$$ 
The second interval, which will be denoted $J'([\alpha])$, is the minimal initial interval containing almost all roots in $\{\pm\alpha\}^E$. (The precise definition will be given shortly after.) It may appear that there are $|\Phi_{W_0}|$ many such intervals. However for different $\alpha$ and $\beta$,
$J(\alpha)$ and $J(\beta)$ (resp. $J'([\alpha])$ and $J'([\beta])$) may be equal. Therefore there are   $m \,(m\leq |\Phi_{W_0}|)$  such initial intervals. We naturally order these intervals together with the empty set and obtain a chain $J(\prec)$. A key ingredient of our proof of the main result is that if $J_2$ is the immediate successor of $J_1$ in $J(\prec)$, the order type of $\prec$ restricted on $J_2\backslash J_1$ is either $\omega$ or $\omega^*+[k]$ for some non-negative integer $k.$

\begin{lemma}
For $\alpha\in \Psi^+_{W_0,\prec}$, there exists $p\geq 0$ such that when $q\geq p$ the closed interval $[(-\alpha)_0+(q+1)\delta, (-\alpha)_0+q\delta]$ under $\prec$ is finite.
\end{lemma}

\begin{proof}
Since $\Phi_{W_0}$ is finite, for an infinite interval $[(-\alpha)_0+(q_i+1)\delta, (-\alpha)_0+q_i\delta]$ there exists $\beta\in \Phi_{W_0}$, $[(-\alpha)_0+(q_i+1)\delta, (-\alpha)_0+q_i\delta]\cap \beta^E$ is infinite. By Lemma \ref{dihedral} we conclude that
$[(-\alpha)_0+(q_i+1)\delta, (-\alpha)_0+q_i\delta]\cap \beta^E=\{(\beta)_0+q\delta|\,q\geq d\}$ for some $d\geq 0.$
Hence at most one such infinite interval $[(-\alpha)_0+(q_i+1)\delta, (-\alpha)_0+q_i\delta]$ has infinite intersection with $\beta^E$.
Again since $\Phi_{W_0}$ is finite, there can only exist finite many $q_i$ such that $[(-\alpha)_0+(q_i+1)\delta, (-\alpha)_0+q_i\delta]$ under $\prec$ is infinite.
\end{proof}

By the above lemma, for $\alpha\in \Psi^+_{W_0,\prec}$, we can define $n(\alpha)$ to be the smallest number such that whenever $q\geq n(\alpha)$ the closed interval $[(-\alpha)_0+(q+1)\delta, (-\alpha)_0+q\delta]$ under $\prec$ is finite. We define an equivalence relation on $\Psi^+_{W_0, \prec}$: $\alpha\sim\beta$ if and only if the nonempty one of the intervals $[(-\alpha)_0+n(\alpha)\delta, (-\beta)_0+n(\beta)\delta]$ and $[(-\beta)_0+n(\beta)\delta,(-\alpha)_0+n(\alpha)\delta]$ under $\prec$ is finite. The equivalence class containing $\alpha$ will be denoted by $[\alpha]$.
Set  $$m([\alpha])=\{\beta\in [\alpha]\mid (-\beta)_0+n(\beta)\delta\succcurlyeq (-\gamma)_0+n(\gamma)\delta, \forall \gamma\in [\alpha]\}.$$

For $\alpha\in \Psi^+_{W_0,\prec}$,
 write $J(\alpha)=\{\gamma\in \Phi^+_{W}\mid\gamma\preccurlyeq (\alpha)_0+n\delta\,\,\text{for}\,\text{some}\,n\geq 0\}$ and $J'([\alpha])=\{\gamma\in \Phi^+_{W}\mid\gamma\preccurlyeq (-m([\alpha]))_0+n(m([\alpha]))\delta\}$. By definition $\{\alpha_0+n\delta|\, n\geq 0\}\subset J(\alpha)$.

Consider the set $$J(\prec)=\{J(\alpha)\mid\,\alpha\in \Psi^+_{W_0,\prec}\}\cup \{J'([\alpha])\mid\,\alpha\in \Psi^+_{W_0,\prec}\}\cup \{\emptyset\}.$$ Clearly $J(\alpha),J'([\alpha])$ are both initial intervals of  the reflection order $(\Phi^+_W,\prec)$. Therefore for two different sets $M,N\in J(\prec)$, we have either $M\subset N$ or $N\subset M$.
Therefore one can totally order the elements in $J(\prec)$ according to containment. This gives rise to a condensation of the reflection order on $\Phi^+_W$ which plays a crucial role in this paper.

\begin{lemma}\label{first}
The minimal element in $J(\prec)\backslash \{\emptyset\}$ is of the form $J(\alpha)$ for some $\alpha\in \Phi_{W_0}$.
\end{lemma}

\begin{proof}
If the minimal element is $J'([\beta])$ for some $\beta\in \Psi^+_{W_0,\prec}$,  then $J(\beta)$  will be properly contained  in $J'([\beta])$. A contradiction.
\end{proof}

\begin{lemma}\label{layer}
\begin{enumerate}
\item Suppose that $K,L\in J(\prec)$ and that $L$  is the (immediate) successor of $K$. Then the interval $L\backslash K$ cannot contain both $(\beta)_0+k\delta$ and $(-\beta)_0+l\delta$ for some $\beta\in \Phi_{W_0}$ and $k,l\geq 0$.

\item Suppose that $J_1\in J(\prec)$ and that $J(\alpha)$ is the (immediate) successor of $J_1$. Then the reflection order $\prec$ restricted on $J(\alpha)\backslash J_1$ has the order type $\omega$.

\item Suppose that $J_1\in J(\prec)$ and that $J'([\alpha])$ is the (immediate) successor of $J_1$. Then the reflection order $\prec$ restricted on $J'([\alpha])\backslash J_1$ has the order type $[n]+\omega^*$ for some $n\geq 0$.

\item The whole positive system $\Phi^+_W$ of the affine Weyl group $W$ is equal to $J'([\alpha])$ for some $\alpha\in \Psi^+_{W_0}$.

\item Suppose that $K,L\in J(\prec)$ and that $L$ be the (immediate) successor of $K$. Then $K_0\subsetneqq L_0$.
\end{enumerate}
\end{lemma}

\begin{proof}
(1)
We assume to the contrary there exists such a root $\beta$.  Without loss of generality we assume that $\beta\in \Psi^+_{W_0,\prec}.$ Then this assumption implies that the set $\{(\beta)_0+i\delta\mid i\geq k\}\cup \{(-\beta)_0+j\delta\mid j\geq l\}$ is contained in $L\backslash K.$  Since $\{(\beta)_0+i\delta\mid i\geq k\}\subset J(\beta)\backslash K$, $J(\beta)$  properly contains $K.$
On the other hand, the set $(-\beta)_0+l\delta\in L\backslash J(\beta)$. Therefore $J(\beta)$ is properly contained in $L$. This contradicts to the fact that $L$ is the immediate successor of $K$.

(2) Let $\gamma+m\delta$ $(\gamma\in \Phi_{W_0})$ be a root in $J(\alpha)\backslash J_1$. Next we show that the  set $I:=\{\zeta\in J(\alpha)\backslash J_1|\,\zeta\preccurlyeq\gamma+m\delta\}$ is finite.
If $I$ is infinite, then there must be some $\lambda\in \Phi_{W_0}$ such that $\lambda^E\cap I$ is infinite since $\Phi_{W_0}$ is finite.
There are two situations.

Case I.  $\lambda^E\cap (J(\alpha)\backslash J_1)=\lambda^E\backslash \{(\lambda)_0, (\lambda)_0+\delta\, \cdots, (\lambda)_0+p\delta\},$$$(\lambda)_0\prec (\lambda)_0+\delta\prec \cdots (-\lambda)_0+2\delta\prec(-\lambda)_0+\delta\prec(-\lambda)_0.$$
 (i.e.  $\lambda\in \Psi_{W_0,\prec}^+$.)
This implies $J(\lambda)$ is  contained in $\{\zeta\in J(\alpha)|\,\zeta\preccurlyeq\gamma+m\delta\}$ and subsequently it is properly contained in
 $J(\alpha)$. But it contains properly $J_1$  as $(\lambda)_0+(p+1)\delta$ is not contained in $J_1$. A contradiction to the fact of $J(\alpha)$ is the successor of $J_1$.

 Case II. $\lambda^E\cap (J(\alpha)\backslash J_1)=\lambda^E\backslash \{(\lambda)_0, (\lambda)_0+\delta\, \cdots, (\lambda)_0+p\delta\},$
$$(-\lambda)_0\prec (-\lambda)_0+\delta\prec \cdots (\lambda)_0+2\delta\prec(\lambda)_0+\delta\prec(\lambda)_0.$$
(i.e.  $-\lambda\in \Psi_{W_0,\prec}^+.$) By definition $(\lambda)_0+(p+1)\delta\prec (\alpha)_0+k\delta$ for some $k\geq 0.$ On the other hand, since $(\lambda)_0+p\delta\not\in J(\alpha),$ $(\alpha)_0+t\delta\prec (\lambda)_0+p\delta$ for all $t\geq 0$. So we conclude that
$[(\lambda)_0+(p+1)\delta, (\lambda)_0+p\delta]$ is an infinite interval under $(\Phi^+_{W},\prec).$
This implies that $n(-\lambda)\geq p+1$ and $(\lambda)_0+n(-\lambda)\delta\in J(\alpha)\backslash J_1.$ Write $\mu=m([-\lambda])$. Note that by definition the closed interval $[(\lambda)_0+n(-\lambda)\delta, (-\mu)_0+n(\mu)\delta]$ is finite.  If $(-\mu)_0+n(\mu)\delta\not\in J(\alpha)\backslash J_1,$ then
all $(\alpha)_0+t\delta, t\geq k$ will be less than $(-\mu)_0+n(\mu)\delta$ and greater than $(\lambda)_0+n(-\lambda)\delta$. A contradiction. So $(-\mu)_0+n(\mu)\delta\in J(\alpha)$ as well.
So $J'([-\lambda])$ contains properly $J_1$ and is properly contained in $J(\alpha)$. A contradiction to the fact of $J(\alpha)$ is the immediate successor of $J_1$.

Now since $(J(\alpha)\backslash J_1)\cap \alpha^E$ is infinite, $J(\alpha)\backslash J_1$ is infinite. Hence the reflection order $\prec$ restricted on $J(\alpha)\backslash J_1$ has the order type $\omega$ by Lemma \ref{nanddualordertype} (1).

(3) Now write $B:=\{\lambda\in \Psi^+_{W_0,\prec}|\,\lambda\in [\alpha]\}$. Consider the set $F:=J'([\alpha])\backslash (J_1\cup (-B)^E)$. Next we show that $F$ is finite.

If to the contrary $F$ is infinite, it must contain some $\gamma\in \Phi_{W_0}$ such that $\gamma^E\cap F$ is infinite since $\Phi_{W_0}$ is finite.  Then there are two cases:

Case I: $\gamma\in \Psi^+_{W_0,\prec}$ and $\gamma^E\cap F=\{(\gamma)_0+k\delta|\,k\geq q\}$ for some $q\geq 0.$ Then $J(\gamma)$ will be greater than $J_1$ but less than $J'([\alpha])$ in $J$. A contradiction.

Case II: $\gamma\in \Psi^-_{W_0,\prec}$ and $\gamma^E\cap F=\{(\gamma)_0+k\delta|\,k\geq q\}$ for some $q\geq 0.$
First we show that $(\gamma)_0+n(-\gamma)\delta$ cannot be greater than $(-m([\alpha]))_0+n(m([\alpha]))\delta$. If $(\gamma)_0+n(-\gamma)\delta\succ (-m([\alpha]))_0+n(m([\alpha]))\delta$,  by the definition of the function $m$, the interval $[(-m([\alpha]))_0+n(m([\alpha]))\delta, (\gamma)_0+n(-\gamma)\delta]$ has to be infinite. But this again contradicts the definition of the function $n$ since  for any $(\gamma)_0+t\delta\in F$ one has  $(\gamma)_0+t\delta\prec (-m([\alpha]))_0+n(m([\alpha]))\delta$.
Therefore $(\gamma)_0+n(-\gamma)\delta\preccurlyeq (-m([\alpha]))_0+n(m([\alpha]))\delta.$

Next we show that  $[(\gamma)_0+n(-\gamma)\delta, (-m([\alpha]))_0+n(m([\alpha]))\delta]$ is finite. If not, by the assumption $\{\gamma+k\delta|\,k\geq n(-\gamma)\}$ will be contained in $J'([\alpha])$. So $J'([-\gamma])$ will be greater than $J_1$ but less than $J'([\alpha])$ in $J(\prec)$. A contradiction.

Hence we conclude $-\gamma\in [\alpha]$. This contradicts the definition of $F$.

Therefore we conclude that $F$ is finite.

Now we show that the reflection order restricted on $J'([\alpha])\backslash (J_1\cup F)$ is of type $\omega^*$.
So take $(\lambda)_0+r\delta$ in this set. By definition $-\lambda\in [\alpha]$. Note that $[(\lambda)_0+r\delta, (\lambda)_0+n(-\lambda)\delta]$ is finite by the definition of the function $n$  and that $[(\lambda)_0+n(-\lambda)\delta, (-m([\alpha]))_0+n(m([\alpha]))\delta]$ is finite since $-\lambda\in [\alpha]$. Therefore we conclude that
$[(\lambda)_0+r\delta, (-m([\alpha]))_0+n(m([\alpha]))\delta]$ is finite. Now the assertion that $J'([\alpha])\backslash (J_1\cup F)$ is of type $\omega^*$ follows form Lemma \ref{nanddualordertype} (1). Now (2) follows from Lemma \ref{nanddualordertype} (2).

(4) It is clear that the maximal element in $J(\prec)$ is of the form $J'([\alpha])$.
Take the maximal element $J'([\alpha])$ in $J(\prec).$
It suffices to show that $n(m([\alpha]))=0$. If not, the interval $[(-m([\alpha]))_0+n(m([\alpha]))\delta, (-m([\alpha]))_0+(n(m([\alpha]))-1)\delta]$ is infinite.
Then by the finiteness of $\Phi_{W_0}$ there exists $\mu\in \Psi^+_{W_0,\prec}$ such that $\{\pm \mu\}^E\cap [(-m([\alpha]))_0+n(m([\alpha]))\delta, (-m([\alpha]))_0+(n(m([\alpha]))-1)\delta]$ is infinite. Then $J'([\mu])$ properly contains $J'([\alpha])$, contradicting the maximality of $J'([\alpha])$.

(5) Clearly $K_0$ is contained in $L_0$. By (2) and (3) $L\backslash K$ is infinite. Since $\Phi_{W_0}$ is finite, there must be some $\lambda\in \Phi_{W_0}$ such that $\lambda^E\cap (L\backslash K)$ is infinite. Therefore the containment is strict.
\end{proof}

Write $J(\prec)=\{J_1, J_2,\cdots J_k\}$ such that $J_1\subset J_2 \subset \cdots \subset J_k$. We have a condensation of $\prec$: $\{J_2\backslash J_1, \cdots J_k\backslash J_{k-1}\}$ which plays a key role in our analysis.
The above lemma has already provided much information about the order type on each $J_i\backslash J_{i-1}$.

\subsection{Example}

In this subsection, we use a concrete example to illustrate our condensation of a reflection order on an irreducible affine Weyl group. For an element $w$ in a Coxeter group $W$, we denote the inversion set of $w^{-1}$
 (i.e. $\{\alpha\in \Phi_W^+\mid w^{-1}(\alpha)\in \Phi_W^-\}$) by $\Phi_w$. Assume that $W$ is infinite. An infinite reduced word of $W$ is a sequence $s_1s_2\cdots, s_i\in S$ (where $S$ is the set of simple reflections) such that for any $k, s_1s_2\cdots s_k$ is reduced. One can define the inversion set $\Phi_{s_1s_2\cdots}=\cup_{i=1}^{\infty}\Phi_{s_1s_2\cdots s_i}.$ Inversion sets are clearly biclosed in $\Phi^+_W$.
 An infinite reduced word with the property that $s_{i}=s_j$ if $i\equiv j(\mathrm{mod}\, k)$  will be written as $(s_1s_2\cdots s_k)^{\infty}$.

Let $W$ be of type $\widetilde{A}_3$. The simple roots of $\Phi_{W_0}$ will be denoted by $\alpha,\beta,\gamma$ with $(\alpha,\beta)=(\beta,\gamma)=-\frac{1}{2}$ and $(\alpha,\gamma)=0$.
Denote by $\Phi_{W_0}^+$ the standard positive system determined by the simple system $\{\alpha,\beta,\gamma\}.$
Then one can check that $((\Phi_{W_0}^+)_{\{\alpha,\beta\},\emptyset})^E$ is the inversion set $\Phi_{u_1}$ where $u_1=(s_{\gamma}s_{\beta}s_{\alpha}s_{\delta-\alpha-\beta-\gamma})^{\infty}$.
Let $u_{1,i}$ be the length $i$ prefix of $(s_{\gamma}s_{\beta}s_{\alpha}s_{\delta-\alpha-\beta-\gamma})^{\infty}$. Write $B_i=\Phi_{u_{1,i}}$.

Denote by $u_{2,i}$ the length $i$ prefix of the infinite reduced word $(s_{\beta}s_{\alpha}s_{\delta-\alpha-\beta})^{\infty}$ of the Coxeter system $(U, \{s_{\alpha},s_{\beta},s_{\delta-\alpha-\beta}\})$ where $U$ is the reflection subgroup generated by $\{s_{\alpha},s_{\beta},s_{\delta-\alpha-\beta}\}$.
Regard $u_{2,i}$ as an element in $W$. Write $C_i=((\Phi_{W_0}^+)_{\{\alpha,\beta\},\emptyset})^E\cup (\Phi_{u_{2,i}}\cap \{\pm\alpha,\pm\beta,\pm(\alpha+\beta)\}^E).$
One can check that $$\Phi_{(s_{\beta}s_{\alpha}s_{\delta-\alpha-\beta})^{\infty}}\cap \{\pm\alpha,\pm\beta,\pm(\alpha+\beta)\}^E=\{\beta,\beta+\alpha\}^E.$$
Therefore
$$((\Phi_{W_0}^+)_{\{\alpha\},\emptyset})^E=((\Phi_{W_0}^+)_{\{\alpha,\beta\},\emptyset})^E\cup (\Phi_{(s_{\beta}s_{\alpha}s_{\delta-\alpha-\beta})^{\infty}}\cap \{\pm\alpha,\pm\beta,\pm(\alpha+\beta)\}^E).$$

Denote by $u_{3,i}$ the length $i$ prefix of the infinite reduced word $(s_{\delta-\beta-\gamma}s_{\gamma}s_{\beta})^{\infty}$ of the Coxeter system $(U, \{s_{\gamma},s_{\beta},s_{\delta-\beta-\gamma}\})$ where $U$ is the reflection subgroup generated by $\{s_{\gamma},s_{\beta},s_{\delta-\beta-\gamma}\}$.
Denote by $D_i$ the set $$((\Phi_{W_0}^+)_{\emptyset,\{\beta,\gamma\}})^E\backslash (\Phi_{u_{3,i}}\cap \{\pm\beta,\pm(\beta+\gamma),\pm\gamma\}^E).$$
One can check that $$\Phi_{(s_{\delta-\beta-\gamma}s_{\gamma}s_{\beta})^{\infty}}\cap \{\pm\gamma,\pm\beta,\pm(\gamma+\beta)\}^E=\{-\beta,-\beta-\gamma\}^E.$$
Therefore
$$(\Phi_{W_0}^+)_{\emptyset,\{\gamma\}}^E=(\Phi_{W_0}^+)_{\emptyset,\{\beta,\gamma\}}^E\backslash (\Phi_{(s_{\delta-\beta-\gamma}s_{\gamma}s_{\beta})^{\infty}}\cap \{\pm\gamma,\pm\beta,\pm(\gamma+\beta)\}^E).$$

Denote by $u_{4,i}$ the length $i$ prefix of the infinite reduced word $(s_{\delta-\alpha-\beta-\gamma}s_{\gamma}s_{\beta}s_{\alpha})^{\infty}$.
Write $E_i=\Phi_{W_0}^E\backslash \Phi_{u_{4,i}}.$
One can check that $$\Phi_{(s_{\delta-\alpha-\beta-\gamma}s_{\gamma}s_{\beta}s_{\alpha})^{\infty}}\cap \Phi_{W_0}^E=\{-\alpha,-\alpha-\beta,-\alpha-\beta-\gamma\}^E.$$
Therefore
$$((\Phi_{W_0}^+)_{\emptyset,\{\beta,\gamma\}})^E=\Phi_{W_0}^E\backslash (\Phi_{(s_{\delta-\alpha-\beta-\gamma}s_{\gamma}s_{\beta}s_{\alpha})^{\infty}}\cap \Phi_{W_0}^E).$$

Consider the following (maximal) chain of biclosed sets in $\Phi_{W_0}^E(=\Phi^+_W):$

$$\emptyset\subset B_1\subset B_2\subset \cdots B_i\subset \cdots$$
$$\subset((\Phi_{W_0}^{+})_{\{\alpha,\beta\},\emptyset})^E\subset C_1\subset C_2\subset\cdots \subset C_j\subset\cdots$$
$$\subset ((\Phi_{W_0}^{+})_{\{\alpha\},\emptyset})^E\subset ((\Phi_{W_0}^{+})_{\{\alpha\},\emptyset})^E\cup \{\alpha\}\subset ((\Phi_{W_0}^{+})_{\{\alpha\},\emptyset})^E\cup \{\alpha,\alpha+\delta\}$$
$$\subset\cdots\subset(((\Phi_{W_0}^+)_{\{\alpha\},\{\gamma\}})^E\backslash{\{-\gamma+k\delta,\dots,-\gamma+3\delta,-\gamma+2\delta,-\gamma+\delta\}})\cup \{\alpha,\alpha+\delta\}$$
$$\subset\cdots\subset(((\Phi_{W_0}^+)_{\{\alpha\},\{\gamma\}})^E\backslash{\{-\gamma+3\delta,-\gamma+2\delta,-\gamma+\delta\}})\cup \{\alpha,\alpha+\delta\}$$
$$\subset(((\Phi_{W_0}^+)_{\{\alpha\},\{\gamma\}})^E\backslash{\{-\gamma+2\delta,-\gamma+\delta\}})\cup \{\alpha,\alpha+\delta\}$$
$$\subset(((\Phi_{W_0}^+)_{\{\alpha\},\{\gamma\}})^E\backslash{\{-\gamma+\delta\}})\cup \{\alpha,\alpha+\delta\}\subset(((\Phi_{W_0}^+)_{\{\alpha\},\{\gamma\}})^E\backslash{\{-\gamma+\delta\}})\cup \{\alpha,\alpha+\delta,\alpha+2\delta\}$$
$$\subset(((\Phi_{W_0}^+)_{\{\alpha\},\{\gamma\}})^E\backslash{\{-\gamma+\delta\}})\cup \{\alpha,\alpha+\delta,\alpha+2\delta,\alpha+3\delta\}\subset$$
$$\cdots\subset((\Phi_{W_0}^+)_{\{\alpha\},\{\gamma\}})^E\backslash{\{-\gamma+\delta\}}\cup \{\alpha,\alpha+\delta,\alpha+2\delta,\alpha+3\delta,\dots,\alpha+l\delta\}\subset$$
$$\cdots\subset ((\Phi_{W_0}^+)_{\emptyset,\{\gamma\}})^E\backslash{\{-\gamma+\delta\}}$$
$$\subset((\Phi_{W_0}^+)_{\emptyset,\{\gamma\}})^E\subset\cdots\subset D_m\subset\cdots \subset D_2\subset D_1\subset((\Phi_{W_0})^+_{\emptyset,\{\beta,\gamma\}})^E$$
$$\subset \cdots \subset E_l \subset\cdots \subset E_2\subset E_1\subset\Phi_{W_0}^E.$$

This  maximal chain of biclosed sets in $\Phi_W^+$ consists precisely of  all initial intervals of
a reflection order on $\Phi_W^+.$ Such a reflection order can be obtained as: $\alpha\prec \beta$ if and only if there exists a biclosed set $B$ in this maximal chain such that $\alpha\in B$ and $\beta\not\in B.$
The  reflection order on $\Phi_W^+$ determined by the above maximal chain of biclosed sets is

$$\gamma\prec \beta+\gamma\prec \alpha+\beta+\gamma\prec \gamma+\delta\prec \beta+\gamma+\delta\prec \alpha+\beta+\gamma+\delta\prec$$
$$\gamma+2\delta\prec \beta+\gamma+2\delta\prec \alpha+\beta+\gamma+2\delta\prec \cdots$$
$$\prec \gamma+i\delta\prec \beta+\gamma+i\delta\prec \alpha+\beta+\gamma+i\delta\prec\cdots$$
$$\prec\beta\prec \alpha+\beta\prec \beta+\delta\prec \alpha+\beta+\delta\prec \beta+2\delta\prec \alpha+\beta+2\delta\prec \cdots$$
$$\prec \beta+j\delta\prec \alpha+\beta+j\delta\prec \cdots\prec \alpha\prec \alpha+\delta$$
$$\prec \cdots \prec -\gamma+k\delta\prec \cdots \prec -\gamma+3\delta\prec -\gamma+2\delta\prec $$
$$\alpha+2\delta\prec \alpha+3\delta\prec \alpha+4\delta\prec \cdots\prec \alpha+l\delta\prec \cdots\prec -\gamma+\delta\prec$$
$$\cdots\prec-\beta+m\delta\prec -\beta-\gamma+m\delta\prec\cdots$$
$$\cdots\prec -\beta+3\delta\prec -\beta-\gamma+3\delta\prec-\beta+2\delta\prec -\beta-\gamma+2\delta\prec -\beta+\delta\prec -\beta-\gamma+\delta\prec$$
$$\cdots\prec-\alpha+p\delta\prec -\alpha-\beta+p\delta\prec -\alpha-\beta-\gamma+p\delta\cdots$$
$$\cdots \prec -\alpha+3\delta\prec -\alpha-\beta+3\delta\prec -\alpha-\beta-\gamma+3\delta \prec -\alpha+2\delta\prec -\alpha-\beta+2\delta$$
$$\prec -\alpha-\beta-\gamma+2\delta  \prec -\alpha+\delta\prec -\alpha-\beta+\delta\prec -\alpha-\beta-\gamma+\delta.$$
This reflection order has the order type $\omega+\omega+[2]+\omega^*+\omega+[1]+\omega^*+\omega^*.$

For this reflection order $\prec$, the positive system $\Psi^+_{W_0,\prec}$ is $\Phi^+_{W_0}=\{\alpha,\beta,\gamma,\alpha+\beta,\beta+\gamma,\alpha+\beta+\gamma\}$.
One sees that $$ J(\gamma)=J(\beta+\gamma)=J(\alpha+\beta+\gamma)=((\Phi_{W_0}^{+})_{\{\alpha,\beta\},\emptyset})^E,$$
$$J(\beta)=J(\alpha+\beta)=((\Phi_{W_0}^{+})_{\{\alpha\},\emptyset})^E, J(\alpha)=((\Phi_{W_0}^+)_{\emptyset,\{\gamma\}})^E\backslash{\{-\gamma+\delta\}}, $$ $$J'([\lambda])=((\Phi_{W_0}^+)_{\{\alpha\},\{\gamma\}})^E\backslash{\{-\gamma+\delta\}}\cup \{\alpha,\alpha+\delta\},$$
$$J'([\beta+\gamma])=((\Phi_{W_0})^+_{\emptyset,\{\beta,\gamma\}})^E,[\beta]=[\beta+\gamma],$$
$$J'([\alpha])=\Phi_{W_0}^E,[\alpha]=[\alpha+\beta]=[\alpha+\beta+\gamma].$$

\section{Main Result}\label{main}

In the previous section, we obtained a condensation $J(\prec)$ of the reflection order $\prec$. In this section we will show that the map
$\Gamma\mapsto \Gamma_0$ converts such a condensation into a chain of biclosed sets in $\Phi_{W_0}$ with a salient feature. We then associate a signature to such a chain of biclosed sets and show that those signatures are precisely trimmed Dyck words.

Let $\Gamma_1, \Gamma_2$ be two different biclosed sets in $\Phi_{W_0}$ and assume that $\Gamma_1\subset \Gamma_2$.
We call the pair $(\Gamma_1, \Gamma_2)$ \emph{admissible} if either  $\Gamma_2\backslash \Gamma_1\subset -\Gamma_1$ or $(\Gamma_2\backslash \Gamma_1)\cap -\Gamma_1=\emptyset.$
Denote the set of  biclosed sets in $\Phi_{W_0}$ by  $\mathcal{B}(\Phi_{W_0})$. Under inclusion they form a poset.
A chain $\emptyset=B_0\subsetneqq B_1\subsetneqq B_2 \subsetneqq \cdots \subsetneqq B_{k-1}\subsetneqq B_k=\Phi_{W_0}$ in the poset $(\mathcal{B}(\Phi_{W_0}),\subset)$ is called admissible if
 every pair $(B_{i-1},B_i), 1\leq i\leq k$ is admissible.

By Corollary 3.27 of \cite{DyerReflOrder}, a chain is admissible if and only if there exists a positive system $\Lambda^+$ in $\Phi_{W_0}$ with its simple system $\Delta$ such that

(1) $B_i=\Lambda^+_{\Delta_{i,1},\Delta_{i,2}}, 0\leq i\leq k$,

(2) $\Delta_{0,1}=\Delta, \Delta_{0,2}=\emptyset, \Delta_{k,1}=\emptyset, \Delta_{k,2}=\Delta$ and

(3) either $\Delta_{i-1,1}=\Delta_{i,1}$ or $\Delta_{i-1,2}=\Delta_{i,2}$.

Note that $\Delta_{i-1,1}\supseteqq\Delta_{i,1}$ and $\Delta_{i-1,2}\subseteqq\Delta_{i,2}$ (exactly one of the containment is strict).

Given an admissible chain $\{B_i\}_{i=0}^k$, we define its \emph{signature} to be a finite sequence $\{\epsilon_i\}_{i=0}^{k-1}$ where $\epsilon_i=1$ if   $B_{i+1}\backslash B_i\subset -B_i$ and $\epsilon=0$ if $(B_{i+1}\backslash B_i)\cap -B_i=\emptyset.$

Let $\prec$ be a reflection order of an affine Weyl group $W$. Write $J(\prec)_0:=\{K_0\mid K\in J(\prec)\}$.

\begin{lemma}\label{basiccorrespond}
\begin{enumerate}
\item Assume that $J,K\in J(\prec)$ and that $J$ is the (immediate) predecessor of $K$. Then the pair $(J_0,K_0)$ is admissible.
\item For any reflection order $\prec$ of $W,$ the set $J(\prec)_0$ ordered by inclusion is an  admissible chain in $(\mathcal{B}(\Phi_{W_0}),\subset)$.
Furthermore if the signature of this chain is $\{\epsilon_i\}_{i=0}^{k-1}$, the order type of $\prec$ is $\sum_{i=0}^{k-1}\lambda_i$ where $\lambda_i=\omega$ if $\epsilon_i=0$ and
$\lambda_i=[n_i]+\omega^*$  for some $n_i\geq 0$ if $\epsilon_i=1.$
\item Let $\{B_i\}_{i=0}^{k}$ be an admissible chain of biclosed sets in $\Phi_{W_0}$. Then there exists a reflection order $\prec$ such that $J(\prec)_0=\{B_i\}_{i=0}^{k}.$
\end{enumerate}
\end{lemma}

\begin{proof}
(1) Assume to the contrary that $(J_0,K_0)$ is not admissible. Then by Corollary 3.27 of \cite{DyerReflOrder}, one can assume that
$J_0=\Lambda^+_{\Delta_1,\Delta_2}, K_0=\Lambda^+_{\Delta_3,\Delta_4}$ where $\Delta_2\subsetneqq \Delta_4$ and $\Delta_3\subsetneqq \Delta_1.$

Case I. One can find $\beta\in (\Delta_1\backslash \Delta_3)\cap (\Delta_4\backslash \Delta_2).$

Then $\beta^E\cap K$ is infinite and $(-\beta)^E\cap K$ is infinite.
$\beta^E\cap J$ is finite and $(-\beta)^E\cap J$ is finite. This is impossible by Lemma \ref{layer} (1).

Case II. The set $(\Delta_1\backslash \Delta_3)\cap (\Delta_4\backslash \Delta_2)$ is empty.

Take $\beta\in (\Delta_1\backslash \Delta_3)$ and $\gamma\in (\Delta_4\backslash \Delta_2).$
(Note that since $\beta\in \Delta_1$, one has $\beta\not\in \Delta_2$. Also since $\gamma\in \Delta_4,$ one has $\gamma\not\in \Delta_3$. So we conclude that $\beta\in \Delta\backslash (\Delta_3\cup\Delta_4), \gamma\in \Delta\backslash (\Delta_1\cup\Delta_2)$ where $\Delta$ is the simple system of $\Lambda^+.$)

This implies that under $\prec$ the elements in $\beta^E$ (resp. $\gamma^E$) appear before the elements in $(-\beta)^E$ (resp. $(-\gamma)^E$). Therefore $\beta,\gamma\in \Psi^+_{W_0,\prec}$.
Furthermore, one has $\beta^E\subset K$, $\gamma^E\subset J$, $(-\gamma)^E\cap K$ is infinite, $(-\gamma)^E\cap J$ is empty, $\beta^E\cap J$ is finite and
$(-\beta)^E\cap K$ is empty.
Therefore one has $J\subsetneqq J(\beta)\subset K$. This forces that $J(\beta)=K$.
By Lemma \ref{layer} (2), restricted on $K\backslash J$, $\prec$ has the order type of $\omega.$ However $(K\backslash J)\cap (-\gamma)^E=\{(-\gamma)_0+t\delta|\,t\geq k\}$ for some $k\geq 0.$ The restriction of $\prec$ on $(K\backslash J)\cap (-\gamma)^E$ should have the order type $\omega^*.$ But this is a subset of $K\backslash J$. A contradiction.

(2) The first assertion follows from (1). Suppose that $J,K\subset J(\prec)$ with $J_0=\Lambda^+_{\Delta_1,\Delta_2}, K_0=\Lambda^+_{\Delta_3,\Delta_2}, \Delta_3\subset \Delta_1.$ Take $\beta\in \Delta_1\backslash \Delta_3.$ Similar to the reasoning in the proof of (1), one deduces that $\beta\in \Psi^+_{W_0,\prec}$ and $J(\beta)=K$. Then by Lemma \ref{layer} (2), restricted on $K\backslash J$, $\prec$ has the order type of $\omega.$  Suppose that $J,K\subset J(\prec)$ with $J_0=\Lambda^+_{\Delta_1,\Delta_2}, K_0=\Lambda^+_{\Delta_1,\Delta_3}, \Delta_3\supset \Delta_2.$ Take $\gamma\in \Delta_3\backslash \Delta_2.$ Then $\gamma\in \Psi^+_{W_0,\prec}$, $\gamma^E\subset J$ and $(-\gamma)^E\cap K$ is infinite. Therefore the order type of $\prec$ restricted on $(-\gamma)^E\cap (K\backslash J)$ is $\omega^*$. Then by Lemma \ref{layer} (2) and (3), the order type of $\prec$ restricted on $K\backslash J$ has to be $[n_i]+\omega^*$ for some $n_i\geq 0$.

(3) Assume that $B_i=\Lambda^+_{\Delta_{i,1},\Delta_{i,2}}, 0\leq i\leq k$. We define a total order on the set $D_i:=B_{i+1}^E\backslash B_{i}^E$.

Case I. $\Delta_{i,2}=\Delta_{i+1,2}$.

It follows from Theorem 3.12 of \cite{afflattice} that $D_i=\Phi_w\cap \Phi_{\Delta_{i,1}}^E$ where $w$ is an infinite reduced word of the reflection subgroup $\widetilde{W_{\Delta_{i,1}}}$ with Coxeter generators $\widetilde{\Theta}$ where $\widetilde{\Theta}$ is the union of the simple system of the positive system $\Phi_{\Delta_{i,1}}\cap \Phi^+_{W_0}$ and $\{\delta-\rho\mid\rho$ is the highest root of an irreducible component of $\Phi_{\Delta_{i,1}}\}$.

Let $r_1r_2\cdots$ be a reduced expression of $w$ (which means $r_i\in \widetilde{\Theta}$). Define a total order $\prec_i$ on $D_i$ by
$$\alpha_{r_1}\prec_i r_1(\alpha_{r_2})\prec_i \cdots \prec_i r_1r_2\cdots r_i(\alpha_{r_{i+1}})\prec_i \cdots.$$

Case II. $\Delta_{i,1}=\Delta_{i+1,1}$.

Again it follows from Theorem 3.12 of \cite{afflattice} that $D_i=\Phi_{\Delta_{i+1,2}}^E\backslash(\Phi_w\cap \Phi_{\Delta_{i+1,2}}^E)$ where $w$ is an infinite reduced word of the reflection subgroup $\widetilde{W_{\Delta_{i+1,2}}}$ with Coxeter generators $\widetilde{\Theta}$ where $\widetilde{\Theta}$ is the union of the simple system of the positive system $\Phi_{\Delta_{i+1,2}}\cap \Phi^+_{W_0}$ and $\{\delta-\rho\mid\rho$ is the highest root of an irreducible component of $\Phi_{\Delta_{i+1,2}}\}$.

Also let $r_1r_2\cdots$ be a reduced expression of $w$ (which means $r_i\in \widetilde{\Theta}$). Define a total order on $D_i$ by
$$\cdots \prec_i r_1r_2\cdots r_i(\alpha_{r_{i+1}})\prec_i \cdots \prec_i r_1(\alpha_{r_2})\prec_i \alpha_{r_1}.$$

The total order $\prec:=\sum_{i=0}^{k-1}(D_i,\prec_i)$ is a reflection order since each of its initial section is a biclosed set of $\Phi_{W_0}^E$. Such a reflection order $\prec$ satisfies the property that
$J(\prec)_0=\{B_i\}_{i=0}^{k}.$
\end{proof}

Now we investigate the admissible chains of biclosed sets in the finite irreducible crystallographic root systems.

An \emph{extended Dyck word} of length $2n$ is a string of $n$ letters $0$ and $n$ letters $1$ such that in any proper initial segment, the number of the letters $1$ is strictly less than that of the letters $0$. For example, for $n=3$, the extended Dyck words are
$$000111,001011.$$ It is well-known that extended Dyck words of length $2n$  are enumerated by the $(n-1)$-th Catalan number.

Let $w$ be an extended Dyck word of length $2n$. We define a $2n$-\emph{trimmed Dyck word}  to be the string obtained by replacing some (contiguous) substrings of $w$ of the form $\underbrace{\epsilon\epsilon\cdots\epsilon}_{k\,\mathrm{times}}, 1<k$ ($\epsilon\in \{0,1\}$) with $\underbrace{\epsilon\epsilon\cdots\epsilon}_{k'\,\mathrm{times}}, 1<k'<k$. For example, for $n=3$, the $6$-trimmed Dyck words are
$$0111, 00111, 000111, 011, 0011, 00011, 01, 001, 0001, 001011, 01011, 00101,0101.$$

\begin{lemma}\label{wordsandchain}
\begin{enumerate}
\item The set of the signatures of maximal admissible chains of biclosed sets in $\Phi_{W_0}$ is precisely the set of extended Dyck words of length $2n.$
\item The set of the signatures of admissible chains of biclosed sets in $\Phi_{W_0}$ is precisely the set of $2n$-trimmed Dyck words.
\end{enumerate}
\end{lemma}

\begin{proof}
(2) follows immediately from (1). We prove (1).
By Corollary 3.27 of \cite{DyerReflOrder}, a maximal admissible chain of biclosed sets in $\Phi$ takes the form
$\{\Lambda^+_{\Delta_{i,1},\Delta_{i,2}}\}_{i=0}^{2n}$ such that either
 $$\Delta_{i,1}\supset \Delta_{i+1,1}, \Delta_{i,2}=\Delta_{i+1,2},|\Delta_{i,1}|-1=|\Delta_{i+1,1}|$$ or
  $$\Delta_{i,2}\subset \Delta_{i+1,2}, \Delta_{i,1}=\Delta_{i+1,1},|\Delta_{i,2}|+1=|\Delta_{i+1,2}|.$$
(Note that $\Delta_{0,1}=\Delta_{2n,2}=\Delta$ and $\Delta_{0,2}=\Delta_{2n,1}=\emptyset.$)

We prove   that for $1\leq i<2n$, in any $i$-th initial segment of the signature of this maximal chain, the number of $0$ is strictly greater than that of $1$.
Suppose to the contrary that in the $i$-th initial segment of the signature of this maximal chain, the number of $0$ is equal to that of $1$.
This implies that $n-|\Delta_{i,1}|=|\Delta_{i,2}|.$ Then $\Delta_{i,1}\uplus \Delta_{i,2}=\Delta$. This contradicts the fact that
$(\Delta_{i,1},\Delta_{i,2})=0$. Suppose to the contrary that in the $i$-th initial segment of the signature of this maximal chain, the number of $0$ is less than that of $1$.
This implies that $n-|\Delta_{i,1}|<|\Delta_{i,2}|,$ which is impossible.

Now we show inductively that given any extended Dyck word of length $2n$, there exists a maximal admissible chain of biclosed sets whose signature equals to it.
First let $W_0$ be of type $A_n, B_n, C_n, F_4$ and $G_2$. Suppose that $\Lambda^+=\Phi_{W_0}^+,$ the standard positive system. The simple roots $\alpha_1,\alpha_2,\dots,\alpha_n$ are numbered such that $(\alpha_i,\alpha_{i+1})\leq0, 1\leq i\leq n-1$ and $(\alpha_i,\alpha_j)=0$ for $|i-j|\geq 2.$

Note that any extended Dyck word starts with $00$. Set $\Delta_{0,1}=\Delta$, $\Delta_{0,2}=\emptyset$, $\Delta_{1,1}=\Delta\backslash \{\alpha_n\}$ and $\Delta_{1,2}=\emptyset$. Assume that $\Delta_{i,1}=\{\alpha_1,\alpha_2,\dots,\alpha_j\}$ and $\Delta_{i,2}=\{\alpha_l,\alpha_{l+1},\dots,\alpha_n\}$ with $l\geq j+2.$
If the $(i+1)$-th number in this extended Dyck word is $0$, we set $\Delta_{i+1,1}=\{\alpha_1,\alpha_2,\dots,\alpha_{j-1}\}$ and $\Delta_{i+1,2}=\{\alpha_l,\alpha_{l+1},\dots,\alpha_n\}$. Note that if $l=j+2$, in the initial segment of length $i$ of this extended Dyck word, there are $n-j$ $0$s and $n-l+1=n-j-1$ $1$s. Hence the $(i+1)$-th number in this extended Dyck word cannot be $1$. So assume that  $l>j+2$ and the  $(i+1)$-th number in this extended Dyck word is $1$.
Then we set $\Delta_{i+1,1}=\{\alpha_1,\alpha_2,\dots,\alpha_{j}\}$ and $\Delta_{i+1,2}=\{\alpha_{l-1},\alpha_{l},\dots,\alpha_n\}$.
One readily sees that the maximal chain of biclosed sets constructed in this manner has the desired signature.

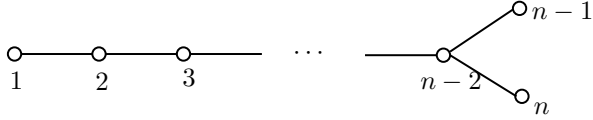
\tikzset{every picture/.style={line width=0.75pt}} 

\begin{tikzpicture}[x=0.75pt,y=0.75pt,yscale=-.6,xscale=.6]

\draw   (138.8,109.6) .. controls (138.8,106.51) and (141.31,104) .. (144.4,104) .. controls (147.49,104) and (150,106.51) .. (150,109.6) .. controls (150,112.69) and (147.49,115.2) .. (144.4,115.2) .. controls (141.31,115.2) and (138.8,112.69) .. (138.8,109.6) -- cycle ;
\draw    (150,109.6) -- (210.3,109.6) ;
\draw   (209.8,109.6) .. controls (209.8,106.51) and (212.31,104) .. (215.4,104) .. controls (218.49,104) and (221,106.51) .. (221,109.6) .. controls (221,112.69) and (218.49,115.2) .. (215.4,115.2) .. controls (212.31,115.2) and (209.8,112.69) .. (209.8,109.6) -- cycle ;
\draw    (221,109.6) -- (281.3,109.6) ;
\draw   (280.8,109.6) .. controls (280.8,106.51) and (283.31,104) .. (286.4,104) .. controls (289.49,104) and (292,106.51) .. (292,109.6) .. controls (292,112.69) and (289.49,115.2) .. (286.4,115.2) .. controls (283.31,115.2) and (280.8,112.69) .. (280.8,109.6) -- cycle ;
\draw    (292,109.6) -- (352.3,109.6) ;
\draw    (439,110.6) -- (499.3,110.6) ;
\draw   (499.3,110.6) .. controls (499.3,107.51) and (501.81,105) .. (504.9,105) .. controls (507.99,105) and (510.5,107.51) .. (510.5,110.6) .. controls (510.5,113.69) and (507.99,116.2) .. (504.9,116.2) .. controls (501.81,116.2) and (499.3,113.69) .. (499.3,110.6) -- cycle ;
\draw    (510.5,107.6) -- (563.3,72.2) ;
\draw   (563.3,71.2) .. controls (563.3,68.11) and (565.81,65.6) .. (568.9,65.6) .. controls (571.99,65.6) and (574.5,68.11) .. (574.5,71.2) .. controls (574.5,74.29) and (571.99,76.8) .. (568.9,76.8) .. controls (565.81,76.8) and (563.3,74.29) .. (563.3,71.2) -- cycle ;
\draw    (510.5,110.6) -- (566.3,145.2) ;
\draw   (565.3,145.2) .. controls (565.3,142.11) and (567.81,139.6) .. (570.9,139.6) .. controls (573.99,139.6) and (576.5,142.11) .. (576.5,145.2) .. controls (576.5,148.29) and (573.99,150.8) .. (570.9,150.8) .. controls (567.81,150.8) and (565.3,148.29) .. (565.3,145.2) -- cycle ;

\draw (376,101) node [anchor=north west][inner sep=0.75pt]   [align=left] {$\displaystyle \cdots $};
\draw (138,122) node [anchor=north west][inner sep=0.75pt]   [align=left] {$1$};
\draw (210,123) node [anchor=north west][inner sep=0.75pt]   [align=left] {$2$};
\draw (283,120) node [anchor=north west][inner sep=0.75pt]   [align=left] {$3$};
\draw (483,122) node [anchor=north west][inner sep=0.75pt]   [align=left] {$n-2$};
\draw (576.8,63.6) node [anchor=north west][inner sep=0.75pt]   [align=left] {$n-1$};
\draw (578.5,148.2) node [anchor=north west][inner sep=0.75pt]   [align=left] {$n$};
\end{tikzpicture}

Now let  $W_0$ be of type $D_n$ and let the numbering of the simple roots of the standard positive system come from  the above Dynkin diagram.
Again note that any extended Dyck word starts with $00$. We set $\Lambda^+=\Phi^+_{W_0}$, the standard positive system, $\Delta_{0,1}=\Delta$, $\Delta_{0,2}=\emptyset$, $\Delta_{1,1}=\Delta\backslash \{\alpha_1\}$ and $\Delta_{1,2}=\emptyset$.

Assume that $\Delta_{i,1}=\{\alpha_{n},\alpha_{n-1},\cdots,\alpha_j\}$ and $\Delta_{i,2}=\{\alpha_1,\alpha_{2},\dots,\alpha_l\}$ with $j\leq n-1$ and $l\leq j-2.$ If the $(i+1)$-th number in this extended Dyck word is $0$, we set $\Delta_{i+1,1}=\{\alpha_n,\alpha_{n+1},\dots,\alpha_{j+1}\}$ and $\Delta_{i+1,2}=\{\alpha_1,\alpha_{2},\dots,\alpha_l\}$. If $l=j-2$, then in the initial segment of length $i$ of this extended Dyck word, there are $j-1$ $0$s and $l=j-2$ $1$s. Hence the $(i+1)$-th number in this extended Dyck word cannot be $1$.
So assume that  $l<j-2$ and $j\leq n-1$ and that the  $(i+1)$-th number in this extended Dyck word is $1$.
Then we set $\Delta_{i+1,1}=\{\alpha_{n},\alpha_{n-1},\dots,\alpha_j\}$ and $\Delta_{i+1,2}=\{\alpha_1,\alpha_{2},\dots,\alpha_{l+1}\}$.

Now assume that $\Delta_{i,1}=\{\alpha_{n}\}$ and $\Delta_{i,2}=\{\alpha_1,\alpha_{2},\cdots,\alpha_l\}, l\leq n-3.$  If the $(i+1)$-th number in this extended Dyck word is $0$, we set $\Delta_{i+1,1}=\emptyset$ and $\Delta_{i+1,2}=\{\alpha_1,\alpha_{2},\dots,\alpha_l\}$. Suppose that the $(i+1)$-th number in this extended Dyck word is $1$. In case $l\leq n-4$, set $\Delta_{i+1,1}=\{\alpha_{n}\}$ and $\Delta_{i+1,2}=\{\alpha_1,\alpha_{2},\dots,\alpha_{l},\alpha_{l+1}\}$. In case $l=n-3$, set $\Delta_{i+1,1}=\{\alpha_{n}\}$ and $\Delta_{i+1,2}=\{\alpha_1,\alpha_{2},\dots,\alpha_{n-3},\alpha_{n-1}\}$.

Now assume that $\Delta_{i,1}=\{\alpha_{n}\}$ and $\Delta_{i,2}=\{\alpha_1,\alpha_{2},\dots,\alpha_{n-3},\alpha_{n-1}\}$. Then in the initial segment of length $i$ of this extended Dyck word, there are $n-1$ 0s and $n-2$ 1s. Hence the $(i+1)$-th number can only be $0$. Set $\Delta_{i+1,1}=\emptyset$ and $\Delta_{i,2}=\{\alpha_1,\alpha_{2},\dots,\alpha_{n-3},\alpha_{n-1}\}$.

Finally assume that $\Delta_{i,1}=\emptyset$. Then in the initial segment of length $i$ of this extended Dyck word there are $n$ 0s and the $(i+1)$-th number in this word can only be $1$. Then set $\Delta_{i+1,2}=\Delta_{i,2}\cup \{\alpha_p\}$ where $\alpha_p\in \Delta\backslash \Delta_{i,2}.$
Then the maximal chain of biclosed sets constructed in the above way has the desired signature.

Type $E_6$

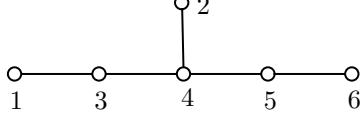
\tikzset{every picture/.style={line width=0.75pt}} 

\begin{tikzpicture}[x=0.75pt,y=0.75pt,yscale=-.6,xscale=.6]

\draw   (138.8,109.6) .. controls (138.8,106.51) and (141.31,104) .. (144.4,104) .. controls (147.49,104) and (150,106.51) .. (150,109.6) .. controls (150,112.69) and (147.49,115.2) .. (144.4,115.2) .. controls (141.31,115.2) and (138.8,112.69) .. (138.8,109.6) -- cycle ;
\draw    (150,109.6) -- (210.3,109.6) ;
\draw   (209.8,109.6) .. controls (209.8,106.51) and (212.31,104) .. (215.4,104) .. controls (218.49,104) and (221,106.51) .. (221,109.6) .. controls (221,112.69) and (218.49,115.2) .. (215.4,115.2) .. controls (212.31,115.2) and (209.8,112.69) .. (209.8,109.6) -- cycle ;
\draw    (221,109.6) -- (281.3,109.6) ;
\draw   (280.8,109.6) .. controls (280.8,106.51) and (283.31,104) .. (286.4,104) .. controls (289.49,104) and (292,106.51) .. (292,109.6) .. controls (292,112.69) and (289.49,115.2) .. (286.4,115.2) .. controls (283.31,115.2) and (280.8,112.69) .. (280.8,109.6) -- cycle ;
\draw    (292,109.6) -- (352.3,109.6) ;
\draw    (285.3,54.2) -- (286.4,104) ;
\draw   (422.3,109.6) .. controls (422.3,106.51) and (424.81,104) .. (427.9,104) .. controls (430.99,104) and (433.5,106.51) .. (433.5,109.6) .. controls (433.5,112.69) and (430.99,115.2) .. (427.9,115.2) .. controls (424.81,115.2) and (422.3,112.69) .. (422.3,109.6) -- cycle ;
\draw   (351.8,109.6) .. controls (351.8,106.51) and (354.31,104) .. (357.4,104) .. controls (360.49,104) and (363,106.51) .. (363,109.6) .. controls (363,112.69) and (360.49,115.2) .. (357.4,115.2) .. controls (354.31,115.2) and (351.8,112.69) .. (351.8,109.6) -- cycle ;
\draw    (362,109.6) -- (422.3,109.6) ;
\draw   (279.3,49.6) .. controls (279.3,46.51) and (281.81,44) .. (284.9,44) .. controls (287.99,44) and (290.5,46.51) .. (290.5,49.6) .. controls (290.5,52.69) and (287.99,55.2) .. (284.9,55.2) .. controls (281.81,55.2) and (279.3,52.69) .. (279.3,49.6) -- cycle ;

\draw (138,122) node [anchor=north west][inner sep=0.75pt]   [align=left] {$1$};
\draw (295,42) node [anchor=north west][inner sep=0.75pt]   [align=left] {$2$};
\draw (209,122) node [anchor=north west][inner sep=0.75pt]   [align=left] {$3$};
\draw (282,120) node [anchor=north west][inner sep=0.75pt]   [align=left] {$4$};
\draw (351.4,122) node [anchor=north west][inner sep=0.75pt]   [align=left] {$5$};
\draw (422,121) node [anchor=north west][inner sep=0.75pt]   [align=left] {$6$};

\end{tikzpicture}

Type $E_7$

\tikzset{every picture/.style={line width=0.75pt}} 

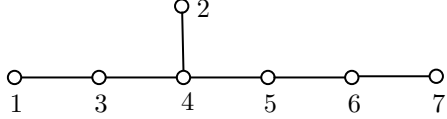
\begin{tikzpicture}[x=0.75pt,y=0.75pt,yscale=-.6,xscale=.6]

\draw   (138.8,109.6) .. controls (138.8,106.51) and (141.31,104) .. (144.4,104) .. controls (147.49,104) and (150,106.51) .. (150,109.6) .. controls (150,112.69) and (147.49,115.2) .. (144.4,115.2) .. controls (141.31,115.2) and (138.8,112.69) .. (138.8,109.6) -- cycle ;
\draw    (150,109.6) -- (210.3,109.6) ;
\draw   (209.8,109.6) .. controls (209.8,106.51) and (212.31,104) .. (215.4,104) .. controls (218.49,104) and (221,106.51) .. (221,109.6) .. controls (221,112.69) and (218.49,115.2) .. (215.4,115.2) .. controls (212.31,115.2) and (209.8,112.69) .. (209.8,109.6) -- cycle ;
\draw    (221,109.6) -- (281.3,109.6) ;
\draw   (280.8,109.6) .. controls (280.8,106.51) and (283.31,104) .. (286.4,104) .. controls (289.49,104) and (292,106.51) .. (292,109.6) .. controls (292,112.69) and (289.49,115.2) .. (286.4,115.2) .. controls (283.31,115.2) and (280.8,112.69) .. (280.8,109.6) -- cycle ;
\draw    (292,109.6) -- (352.3,109.6) ;
\draw    (285.3,54.2) -- (286.4,104) ;
\draw   (422.3,109.6) .. controls (422.3,106.51) and (424.81,104) .. (427.9,104) .. controls (430.99,104) and (433.5,106.51) .. (433.5,109.6) .. controls (433.5,112.69) and (430.99,115.2) .. (427.9,115.2) .. controls (424.81,115.2) and (422.3,112.69) .. (422.3,109.6) -- cycle ;
\draw   (351.8,109.6) .. controls (351.8,106.51) and (354.31,104) .. (357.4,104) .. controls (360.49,104) and (363,106.51) .. (363,109.6) .. controls (363,112.69) and (360.49,115.2) .. (357.4,115.2) .. controls (354.31,115.2) and (351.8,112.69) .. (351.8,109.6) -- cycle ;
\draw    (362,109.6) -- (422.3,109.6) ;
\draw   (279.3,49.6) .. controls (279.3,46.51) and (281.81,44) .. (284.9,44) .. controls (287.99,44) and (290.5,46.51) .. (290.5,49.6) .. controls (290.5,52.69) and (287.99,55.2) .. (284.9,55.2) .. controls (281.81,55.2) and (279.3,52.69) .. (279.3,49.6) -- cycle ;
\draw   (493.3,108.6) .. controls (493.3,105.51) and (495.81,103) .. (498.9,103) .. controls (501.99,103) and (504.5,105.51) .. (504.5,108.6) .. controls (504.5,111.69) and (501.99,114.2) .. (498.9,114.2) .. controls (495.81,114.2) and (493.3,111.69) .. (493.3,108.6) -- cycle ;
\draw    (433,108.6) -- (493.3,108.6) ;

\draw (138,122) node [anchor=north west][inner sep=0.75pt]   [align=left] {$1$};
\draw (295,42) node [anchor=north west][inner sep=0.75pt]   [align=left] {$2$};
\draw (209,122) node [anchor=north west][inner sep=0.75pt]   [align=left] {$3$};
\draw (282,120) node [anchor=north west][inner sep=0.75pt]   [align=left] {$4$};
\draw (351.4,122) node [anchor=north west][inner sep=0.75pt]   [align=left] {$5$};
\draw (422,122) node [anchor=north west][inner sep=0.75pt]   [align=left] {$6$};
\draw (493,122) node [anchor=north west][inner sep=0.75pt]   [align=left] {$7$};

\end{tikzpicture}

Type $E_8$

\tikzset{every picture/.style={line width=0.75pt}} 

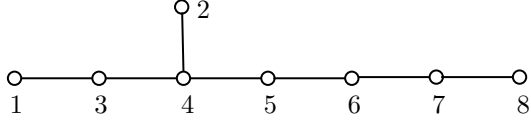
\begin{tikzpicture}[x=0.75pt,y=0.75pt,yscale=-.6,xscale=.6]

\draw   (138.8,109.6) .. controls (138.8,106.51) and (141.31,104) .. (144.4,104) .. controls (147.49,104) and (150,106.51) .. (150,109.6) .. controls (150,112.69) and (147.49,115.2) .. (144.4,115.2) .. controls (141.31,115.2) and (138.8,112.69) .. (138.8,109.6) -- cycle ;
\draw    (150,109.6) -- (210.3,109.6) ;
\draw   (209.8,109.6) .. controls (209.8,106.51) and (212.31,104) .. (215.4,104) .. controls (218.49,104) and (221,106.51) .. (221,109.6) .. controls (221,112.69) and (218.49,115.2) .. (215.4,115.2) .. controls (212.31,115.2) and (209.8,112.69) .. (209.8,109.6) -- cycle ;
\draw    (221,109.6) -- (281.3,109.6) ;
\draw   (280.8,109.6) .. controls (280.8,106.51) and (283.31,104) .. (286.4,104) .. controls (289.49,104) and (292,106.51) .. (292,109.6) .. controls (292,112.69) and (289.49,115.2) .. (286.4,115.2) .. controls (283.31,115.2) and (280.8,112.69) .. (280.8,109.6) -- cycle ;
\draw    (292,109.6) -- (352.3,109.6) ;
\draw    (285.3,54.2) -- (286.4,104) ;
\draw   (422.3,109.6) .. controls (422.3,106.51) and (424.81,104) .. (427.9,104) .. controls (430.99,104) and (433.5,106.51) .. (433.5,109.6) .. controls (433.5,112.69) and (430.99,115.2) .. (427.9,115.2) .. controls (424.81,115.2) and (422.3,112.69) .. (422.3,109.6) -- cycle ;
\draw   (351.8,109.6) .. controls (351.8,106.51) and (354.31,104) .. (357.4,104) .. controls (360.49,104) and (363,106.51) .. (363,109.6) .. controls (363,112.69) and (360.49,115.2) .. (357.4,115.2) .. controls (354.31,115.2) and (351.8,112.69) .. (351.8,109.6) -- cycle ;
\draw    (362,109.6) -- (422.3,109.6) ;
\draw   (279.3,49.6) .. controls (279.3,46.51) and (281.81,44) .. (284.9,44) .. controls (287.99,44) and (290.5,46.51) .. (290.5,49.6) .. controls (290.5,52.69) and (287.99,55.2) .. (284.9,55.2) .. controls (281.81,55.2) and (279.3,52.69) .. (279.3,49.6) -- cycle ;
\draw   (493.3,108.6) .. controls (493.3,105.51) and (495.81,103) .. (498.9,103) .. controls (501.99,103) and (504.5,105.51) .. (504.5,108.6) .. controls (504.5,111.69) and (501.99,114.2) .. (498.9,114.2) .. controls (495.81,114.2) and (493.3,111.69) .. (493.3,108.6) -- cycle ;
\draw    (433,108.6) -- (493.3,108.6) ;
\draw   (563.3,108.6) .. controls (563.3,105.51) and (565.81,103) .. (568.9,103) .. controls (571.99,103) and (574.5,105.51) .. (574.5,108.6) .. controls (574.5,111.69) and (571.99,114.2) .. (568.9,114.2) .. controls (565.81,114.2) and (563.3,111.69) .. (563.3,108.6) -- cycle ;
\draw    (503,108.6) -- (563.3,108.6) ;

\draw (138,122) node [anchor=north west][inner sep=0.75pt]   [align=left] {$1$};
\draw (295,42) node [anchor=north west][inner sep=0.75pt]   [align=left] {$2$};
\draw (209,122) node [anchor=north west][inner sep=0.75pt]   [align=left] {$3$};
\draw (282,122) node [anchor=north west][inner sep=0.75pt]   [align=left] {$4$};
\draw (351.4,122) node [anchor=north west][inner sep=0.75pt]   [align=left] {$5$};
\draw (422,122) node [anchor=north west][inner sep=0.75pt]   [align=left] {$6$};
\draw (493,122) node [anchor=north west][inner sep=0.75pt]   [align=left] {$7$};
\draw (564,122) node [anchor=north west][inner sep=0.75pt]   [align=left] {$8$};

\end{tikzpicture}

Now let  $W_0$ be of type $E_n, n=6,7,8$ and let the numbering of the simple roots of the standard positive system come from  the above Dynkin diagram.
Again since any extended Dyck word starts with $00$, we set $\Lambda^+=\Phi^+_{W_0}$, the standard positive system, $\Delta_{0,1}=\Delta$, $\Delta_{0,2}=\emptyset$, $\Delta_{1,1}=\Delta\backslash \{\alpha_n\}$ and $\Delta_{1,2}=\emptyset$.

Assume that $\Delta_{i,1}=\{\alpha_{1},\alpha_{2},\dots,\alpha_j\}$ and $\Delta_{i,2}=\{\alpha_{n},\alpha_{n-1},\dots,\alpha_l\}$ with $j\geq 4$ and $l\geq j+2.$ If the $(i+1)$-th number in this extended Dyck word is $0$, we set $\Delta_{i+1,1}=\{\alpha_{1},\alpha_{2},\dots,\alpha_{j-1}\}$ and $\Delta_{i+1,2}=\{\alpha_{n},\alpha_{n-1},\dots,\alpha_l\}$. If $l=j+2$, in the initial segment of length $i$ of this extended Dyck word, there are $n-j$ $0$s and $n-l+1=n-j-1$ $1$s. Hence the $(i+1)$-th number in this extended Dyck word cannot be $1$.
So assume that  $l>j+2$  and that the $(i+1)$-th number in this extended Dyck word is $1$.
Then we set $\Delta_{i+1,1}=\{\alpha_{1},\alpha_{2},\dots,\alpha_j\}$ and $\Delta_{i+1,2}=\{\alpha_{n},\alpha_{n-1},\dots,\alpha_{l-1}\}$.

Assume that $\Delta_{i,1}=\{\alpha_{1},\alpha_{2},\alpha_3\}$ and $\Delta_{i,2}=\{\alpha_{n},\alpha_{n-1},\dots,\alpha_l\}$ with $l\geq 6.$ Suppose that the
$(i+1)$-th number in this extended Dyck word is $0$. Set $\Delta_{i+1,1}=\{\alpha_{1},\alpha_3\}$ and $\Delta_{i+1,2}=\{\alpha_{n},\alpha_{n-1},\dots,\alpha_l\}$.
Suppose that the
$(i+1)$-th number in this extended Dyck word is $1$. Then we set $\Delta_{i+1,1}=\{\alpha_{1},\alpha_2,\alpha_3\}$ and $\Delta_{i+1,2}=\{\alpha_{n},\alpha_{n-1},\dots,\alpha_{l-1}\}$.

Assume that $\Delta_{i,1}=\{\alpha_{1},\alpha_{2},\alpha_3\}$ and $\Delta_{i,2}=\{\alpha_{n},\alpha_{n-1},\dots,\alpha_5\}$. Then the
$(i+1)$-th number in this extended Dyck word can only be $0$ since  there are $n-3$ 0s and $n-4$ 1s in the length $i$ initial segment of this extended Dyck word.
Then we set $\Delta_{i+1,1}=\{\alpha_{1},\alpha_3\}$ and $\Delta_{i+1,2}=\{\alpha_{n},\alpha_{n-1},\dots,\alpha_5\}$.

Assume that $\Delta_{i,1}=\{\alpha_{1},\alpha_3\}$ and $\Delta_{i,2}=\{\alpha_{n},\alpha_{n-1},\dots,\alpha_l\}$ with $l\geq 5$. If the
$(i+1)$-th number in this extended Dyck word is $0$, we set $\Delta_{i+1,1}=\{\alpha_{1}\}$ and $\Delta_{i+1,2}=\{\alpha_{n},\alpha_{n-1},\dots,\alpha_l\}$.
Assume instead that the
$(i+1)$-th number in this extended Dyck word is $1$. In case $l>5$, we set $\Delta_{i+1,1}=\{\alpha_{1},\alpha_3\}$ and $\Delta_{i+1,2}=\{\alpha_{n},\alpha_{n-1},\dots,\alpha_{l-1}\}.$
In case $l=5$, we set $\Delta_{i+1,1}=\{\alpha_{1},\alpha_3\}$ and $\Delta_{i+1,2}=\{\alpha_{n},\alpha_{n-1},\dots,\alpha_{5},\alpha_2\}.$

Assume that $\Delta_{i,1}=\{\alpha_{1},\alpha_3\}$ and $\Delta_{i,2}=\{\alpha_{n},\alpha_{n-1},\dots,\alpha_{5},\alpha_2\}$. In this case the
$(i+1)$-th number in this extended Dyck word can only be $0$ since  there are $n-2$ 0s and $n-3$ 1s in the length $i$ initial segment of this extended Dyck word.
Then we set $\Delta_{i+1,1}=\{\alpha_{1}\}$ and $\Delta_{i+1,2}=\{\alpha_{n},\alpha_{n-1},\dots,\alpha_{5},\alpha_2\}.$

Assume that $\Delta_{i,1}=\{\alpha_{1}\}$ and $\Delta_{i,2}=\{\alpha_{n},\alpha_{n-1},\dots,\alpha_l\}$ with $l\geq 4$. If the
$(i+1)$-th number in this extended Dyck word is $0$, we set $\Delta_{i+1,1}=\emptyset$ and $\Delta_{i+1,2}=\{\alpha_{n},\alpha_{n-1},\dots,\alpha_l\}$.
Suppose otherwise that the
$(i+1)$-th number in this extended Dyck word is $1$.  If $l\geq 5,$ we set $\Delta_{i+1,1}=\{\alpha_{1}\}$ and $\Delta_{i+1,2}=\{\alpha_{n},\alpha_{n-1},\dots,\alpha_{l-1}\}$.
If $l=4$, we set $\Delta_{i+1,1}=\{\alpha_{1}\}$ and $\Delta_{i+1,2}=\{\alpha_{n},\alpha_{n-1},\dots,\alpha_{4},\alpha_2\}$.

Assume that $\Delta_{i,1}=\{\alpha_{1}\}$ and $\Delta_{i,2}=\{\alpha_{n},\alpha_{n-1},\dots,\alpha_{4},\alpha_2\}$. Then the $(i+1)$-th number in this extended Dyck word can only be $0$ since there are $n-1$ 0s and $n-2$ 1s in the length $i$ initial segment of this extended Dyck word. Then we set $\Delta_{i+1,1}=\emptyset$ and $\Delta_{i+1,2}=\{\alpha_{n},\alpha_{n-1},\dots,\alpha_{4},\alpha_2\}$.

Assume that $\Delta_{i,1}=\{\alpha_{1}\}$ and $\Delta_{i,2}=\{\alpha_{n},\alpha_{n-1},\dots,\alpha_{5},\alpha_2\}$. If the $(i+1)$-th number in this extended Dyck word is $0$, we set $\Delta_{i+1,1}=\emptyset$ and $\Delta_{i+1,2}=\{\alpha_{n},\alpha_{n-1},\dots,\alpha_{5},\alpha_2\}$. If the $(i+1)$-th number in this extended Dyck word is $1$, we set $\Delta_{i+1,1}=\{\alpha_{1}\}$ and $\Delta_{i+1,2}=\{\alpha_{n},\alpha_{n-1},\dots,\alpha_5, \alpha_{4},\alpha_2\}$.

Finally assume that $\Delta_{i,1}=\emptyset$. Then in the initial segment of  length $i$ of this extended Dyck word there are $n$ 0s. So the $(i+1)$-th number in this word can only be $1$. Then we set $\Delta_{i+1,2}=\Delta_{i,2}\cup \{\alpha_p\}$ where $\alpha_p\in \Delta\backslash \Delta_{i,2}.$ Then the maximal chain of biclosed sets constructed in the above way has the desired signature.
\end{proof}

\begin{theorem}\label{mainthm}
Let $W$ be an irreducible affine Weyl group of rank $n$. A possible order type of the reflection orders on $W$ has the following form.

Take a $2n$-trimmed Dyck word $\epsilon_1\epsilon_2\cdots\epsilon_k$. Construct an order type $\sum_{i=1}^k \lambda_i$ where

$\lambda_i=\omega$ if $\epsilon_i=0$ and $\lambda_i=[n_i]+\omega^*$ if $\epsilon_i=1$ for some nonnegative $n_i$.

Furthermore 

(a) $n_i$ can only be 0 if for all $j\leq i-1$, $\epsilon_j=0$ and for all $k\geq i$, $\epsilon_k=1$.

(b) $n_i$ can be any arbitrary nonnegative integer in other cases.
\end{theorem}

\begin{proof}
By Lemma \ref{basiccorrespond} and Lemma \ref{wordsandchain}, any possible order type of the reflection orders on $W$  takes the form $\sum_{i=1}^k \lambda_i$ such that
$\lambda_i=\omega$ if $\epsilon_i=0$ and
$\lambda_i=[n_i]+\omega^*, n_i\geq 0$ if $\epsilon_1=1.$
We need to prove that the additional (stronger) statements (a) and (b) hold.

 For (a), we show that if for all $j\leq i-1$, $\epsilon_j=0$ and for all $k\geq i$, $\epsilon_k=1$, $\lambda_i$ cannot be
$[p]+\omega^*$ for $p>0.$ Suppose that $\prec$ is a reflection order such that the chain $J(\prec)_0$ has such a signature.
The conditions force that the $i$-th element in $J(\prec)\backslash \{\emptyset\}$ is of the form $J(\alpha)$ and $J(\alpha)_0=\Lambda^+_{\emptyset,\emptyset}=\Lambda^+.$ Therefore $J(\alpha)$ can only be $(\Lambda^+_{\emptyset,\emptyset})^E=(\Lambda^+)^E.$ We conclude that
for all $\gamma\in \Lambda^+$, $\gamma^E\subset J(\alpha).$ Therefore $\Phi_{W_0}^E\backslash J(\alpha)=\cup_{\gamma\in \Lambda^+}(-\lambda)^E$. The restriction of $\prec$ on each $(-\lambda)^E$ is of type $\omega^*$. Therefore there cannot be any finite initial interval for $(\Phi_{W_0}^E\backslash J(\alpha),\prec\mid_{\Phi_{W_0}^E\backslash J(\alpha)})$.

For (b), first we assume that $\epsilon_{i-1}=\epsilon_i=1$. 
One notes that $\omega^*+[n_i]+\omega^*=\omega^*+\omega^*.$ Therefore $n_i$ can be arbitrary natural number.

Next we assume that $\epsilon_{i-1}=0, \epsilon_i=1$ and there exists some $j<i-1$ (resp. $t>i$) such that $\epsilon_j=1$ (resp. $\epsilon_t=0$)
We show that $\lambda_i$ can indeed be $[n_i]+\omega^*$ with $n_i>0$ being prescribed.
For convenience, we call $i$ an insertable index if $\epsilon_{i-1}=0, \epsilon_i=1$ and there exists some $j<i-1$ (resp. $t>i$) such that $\epsilon_j=1$ (resp. $\epsilon_t=0$).
We first consider the collection $\{(\Phi_{W_0}^+)_{\Delta_{i,1},\Delta_{i,2}}\}_{i=0}^k$ which is an admissible chain of biclosed sets in $\Phi_{W_0}$ whose signature is the desired $2n$-trimmed Dyck word.
Then one gets an associated chain of biclosed sets $\{((\Phi_{W_0}^+)_{\Delta_{i,1},\Delta_{i,2}})^E\}_{i=0}^k$ in $\Phi_W^+.$
One can extend this chain of biclosed sets in $\Phi_W^+$ to a maximal chain of biclosed sets in $\Phi^+_W$, which in turn gives rise to a reflection order $\prec.$ In particular, one has
$\{J(\prec)_0\}=\{(\Phi_{W_0}^+)_{\Delta_{i,1},\Delta_{i,2}}\}_{i=0}^k.$ Now suppose that $p$ is the smallest insertable index. Then one has $\epsilon_{p-1}=0$ and $\epsilon_p=1.$
Let $q$ be the smallest index such that $q>p$ and $\epsilon_q=0.$ (Such $q$ exists since the condition in (a) is not satisfied.) Denote by $B_i$ the $i$-th element in the chain $\{((\Phi_{W_0}^+)_{\Delta_{i,1},\Delta_{i,2}})^E\}_{i=0}^k$. (i.e. $B_i=((\Phi_{W_0}^+)_{\Delta_{i,1},\Delta_{i,2}})^E$.) When restricted on $B_{p-1}\backslash B_{p-2}$, the total order $\prec$ takes the form
$$\beta_1\prec \beta_2\prec \beta_3 \prec \cdots. $$  When restricted on $B_{j}\backslash B_{j-1}, p\leq j\leq q-1$, the total order $\prec$ takes the form
$$\cdots \prec\gamma_{j,3}\prec \gamma_{j,2}\prec \gamma_{j,1}.$$ When restricted on $B_{q}\backslash B_{q-1}$, the total order $\prec$ takes the form
$$\mu_1\prec \mu_2\prec \mu_3 \prec \cdots. $$
By our construction the set of roots $\{\mu_1, \cdots, \mu_{n_p}\}$ is contained in
 $$(\Phi_{\Delta_{q-1,1}}^+\backslash \Phi_{\Delta_{q,1}}^+)^E$$
 and the set of roots $\{\gamma_{j,z},p\leq j\leq q-1, 1\leq z\}$ is contained in $$(\Phi_{\Delta_{q-1,2}}^{-}\backslash \Phi_{\Delta_{p-1,2}}^-)^E.$$

One can construct a new reflection order $\prec_1$ by  moving the interval $$\mu_1\prec \mu_2\prec \mu_3 \prec \cdots \prec \mu_{n_p}$$
in-between $\beta_1\prec \beta_2\prec \beta_3 \prec \cdots$ and $\cdots \prec\gamma_{p,3}\prec \gamma_{p,2}\prec \gamma_{p,1}$.
Since $\Delta_{q-1,1}$ and $\Delta_{q-1,2}$ are orthogonal, one readily sees that the initial intervals of $\prec_1$ are all biclosed in $\Phi_{W}^+$ and therefore $\prec_1$ is also a reflection order. Furthermore $\prec_1$ and $\prec$ have the same associated signature (and in fact $J(\prec_1)_0$ and $J(\prec)_0$ are the same)  and $\lambda_p$  is indeed $[n_p]+\omega^*.$ This process can be visualized in the following picture.

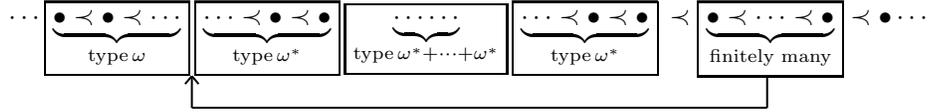
\tikzset{every picture/.style={line width=0.75pt}} 
\begin{tikzpicture}[x=0.8pt,y=0.75pt,yscale=-0.8,xscale=0.8]
\draw (0,0) node [anchor=north west][inner sep=0.75pt]   [align=left] {$\cdots\fbox{$\underbrace{ \bullet \prec \bullet \prec \cdots}_{\text{type}\,\omega}$}\,\fbox{$\underbrace{\cdots \prec\bullet \prec \bullet }_{\text{type}\,\omega^*}$} \,\fbox{$\underbrace{\cdots\cdots}_{\text{type}\,\omega^*+\cdots+\omega^*}$} \, \fbox{$\underbrace{\cdots \prec\bullet \prec \bullet}_{\text{type}\,\omega^*}$}\, \prec \fbox{$\underbrace{\bullet\prec\cdots \prec \bullet}_{\text{finitely many}}$}\prec\bullet\cdots$};
\draw    (450,49) --  (450,69) ;
\draw    (450,69) --  (110,69) ;
\draw    (110,69) --  (110,49) ;
\draw    (110,49) -- (114,54) ;
\draw     (110, 49) -- (106, 54);
\end{tikzpicture}

Let $I=\{1\leq i\leq k-1|\,i\,\text{is}\,\text{insertable}\}$ and let $t$ be the maximum element in $I$.
We can repeat the above operation to create a reflection order $\prec_r$ such that $\lambda_i$  is $[n_i]+\omega^*$ for all $i\in I\backslash\{t\}$ and $\prec_r$ and $\prec$ have the same associated signature. 

Finally we treat the last insertable  position. Now let $B_i$ be the $i-$th element of $J(\prec_r)$.
Assume that $u$ be the greatest index such that $u<t$ and $\epsilon_u=1$.
Restricted on $B_{u}\backslash B_{u-1}$,  $\prec_r$ has the form
$$\cdots \prec\kappa_{3}\prec \kappa_{2}\prec \kappa_{1}.$$
Restricted on $B_{i}\backslash B_{i-1}, u+1\leq i\leq t-1$,  $\prec_r$ has the form
$$\eta_{i,1}\prec \eta_{i,2}\prec \eta_{i,3} \prec \cdots.$$  
$B_t\backslash B_{t-1}$, $\prec_r$ has the form
$$\cdots \prec\upsilon_{3}\prec \upsilon_{2}\prec \upsilon_{1}.$$ 

By our construction the set of roots $\{\kappa_1, \cdots, \kappa_{n_t}\}$ is contained in
$$(\Phi_{\Delta_{u,2}}^-\backslash \Phi_{\Delta_{u-1,2}}^-)^E$$
and the set of roots $\{\eta_{i,1}\prec \eta_{i,2}\prec \eta_{i,3} \prec \cdots\}$ is contained in
$$(\Phi_{\Delta_{u,1}}^+\backslash \Phi_{\Delta_{t,1}}^+)^E.$$
Therefore the set of roots $\{\kappa_1, \cdots, \kappa_{n_t}\}$ and the set of roots $\{\eta_{i,1}\prec \eta_{i,2}\prec \eta_{i,3} \prec \cdots\}$ are orthogonal. So we can construct a new reflection order $\prec_{r+1}$ by  moving the interval $$\kappa_{n_t}\prec\cdots \prec\kappa_{3}\prec \kappa_{2}\prec \kappa_{1}$$
in-between $\eta_{t-1,1}\prec \eta_{t-1,2}\prec \eta_{t-1,3} \prec \cdots$ and $\cdots \prec\upsilon_{3}\prec \upsilon_{2}\prec \upsilon_{1}$.
Then $\prec_{r+1}$ is a reflection with the desired order type. This last step can be visualized in the following picture.

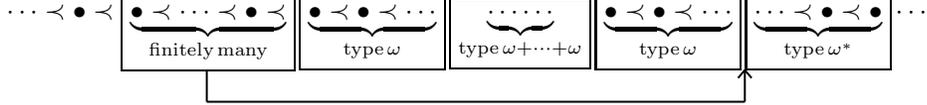
\tikzset{every picture/.style={line width=0.75pt}} 
\begin{tikzpicture}[x=0.8pt,y=0.75pt,yscale=-0.8,xscale=0.8]
\draw (0,0) node [anchor=north west][inner sep=0.75pt]   [align=left] {$\cdots\prec\bullet\prec \fbox{$\underbrace{\bullet \prec\cdots  \prec \bullet\prec}_{\text{finitely}\,\text{many}}$}\,\fbox{$\underbrace{ \bullet \prec \bullet \prec \cdots}_{\text{type}\,\omega}$}\,\fbox{$\underbrace{\cdots\cdots}_{\text{type}\,\omega+\cdots+\omega}$} \,\fbox{$\underbrace{ \bullet \prec \bullet \prec \cdots}_{\text{type}\,\omega}$}\, \fbox{$\underbrace{\cdots \prec\bullet \prec \bullet }_{\text{type}\,\omega^*}$}  \cdots$};
\draw    (120,49) --  (120,69) ;
\draw    (120,69) --  (438,69) ;
\draw    (438,69) --  (438,49) ;
\draw    (438,49)  --  (442,54);
\draw    (438,49)  --  (434,54);
\end{tikzpicture}

\end{proof}

\subsection{Example}\label{keyexample} The order types of the reflection orders on affine Weyl group of rank 3 (e.g.  of type $\widetilde{A}_3$) are as follows:

\begin{tabular}{|c|c|}
  \hline
  6-trimmed Dyck word & Order type(s)\\
  \hline
  $0111$ & $\omega+\omega^*+\omega^*+\omega^*$ \\
  $00111$ & $\omega+\omega+\omega^*+\omega^*+\omega^*$ \\
  $000111$ & $\omega+\omega+\omega+\omega^*+\omega^*+\omega^*$ \\
  $011$  &  $\omega+\omega^*+\omega^*$ \\
  $0011$  &  $\omega+\omega+\omega^*+\omega^*$ \\
  $00011$  &  $\omega+\omega+\omega+\omega^*+\omega^*$ \\
  $01$  &   $\omega+\omega^*$   \\
  $001$  &   $\omega+\omega+\omega^*$   \\
  $0001$  &   $\omega+\omega+\omega+\omega^*$   \\
  $001011$  &  $\omega+\omega+[a]+\omega^*+\omega+[b]+\omega^*+\omega^*$ for some $a,b\geq 0$ \\
  $01011$  &  $\omega+[a]+\omega^*+\omega+[b]+\omega^*+\omega^*$ for some $a,b\geq 0$ \\
  $00101$  &  $\omega+\omega+[a]+\omega^*+\omega+[b]+\omega^*$ for some $a,b\geq 0$ \\
  $0101$  &  $\omega+[a]+\omega^*+\omega+[b]+\omega^*$ for some $a,b\geq 0$ \\
  \hline
\end{tabular}

\section*{appendix: Enumeration of 2n-trimmed Dyck words}\label{appendix}

In this appendix we shall show that  the number of $2n$-trimmed Dyck words is $C_{n+1}-1$. The key idea of the proof is that in order to obtain any $2n$-trimmed Dyck word, it is sufficient to simply shorten the initial segment (consisting of 0s)  and the final segment (consisting of 1s) of all extended Dyck words of length $2n$.

A sequence of $0$s and $1$s will be called a \emph{generalized} $(a, b, c)-$\emph{Dyck sequence} if

(a) it consists of $a$ $0$s and $b$ $1$s  and 

(b) in any initial subsequence of it the number of $1$s minus the number of $0$s cannot  be  greater than or equal to $c$.

A standard Dyck word of length $2n$ is a generalized $(n, n, 1)-$Dyck sequence.

The following result is analogous to the one regarding the enumeration of standard Dyck words of length $2n$.

\begin{theorem}\label{counting}
The number of generalized $(a, b, c)-$Dyck words is given by $$\binom{a+b}{a}-\binom{a+b}{a+c}.$$
\end{theorem} 

\begin{proof}
To prove this assertion, we adopt the similar strategy which proves the fact that Catalan number $\frac{1}{n+1}\binom{2n}{n}$ counts certain balanced paths.
Suppose $s$ is a sequence consisting $a$ $0$s and $b$ $1$s.
From $s$  we can construct a path (consisting $a+b$ steps) which starts at $(0,0)$ and ends at $(a+b,a-b)$.
The $i-$th step of this path starts at the endpoint of the $(i-1)-$th step $(m,n)$. It ends at the point $(m+1, n+1)$ if the $i-$th letter in $s$ is $0$.
It ends at the point $(m+1, n-1)$ if the $i-$th letter in $s$ is $1$. 
Clearly there are $\binom{a+b}{a}$ many such paths. Among these paths, some correspond to generalized $(a, b, c)-$Dyck sequences and some do not.
For those which do not, they must touch the horizontal line $y=-c$ at some point.
One can flip the part of the path after that point about this line $y=-c$. Then those paths which do not correspond to  generalized $(a, b, c)-$Dyck sequences end at $(a+b, b-a-2c).$ Furthermore after reflecting, there are $b-c$  $0$s and $a+c$  $1$s in the path. Therefore there are
$\binom{a+b}{a+c}$  such paths not corresponding to generalized $(a, b, c)-$Dyck sequences. The assertion follows.

\begin{center}
\tikzset{every picture/.style={line width=0.75pt}} 

\begin{tikzpicture}[x=0.75pt,y=0.75pt,yscale=-1,xscale=1]

\draw  [draw opacity=0][dash pattern={on 0.84pt off 2.51pt}] (200.73,1227.16) -- (381.73,1227.16) -- (381.73,1328.16) -- (200.73,1328.16) -- cycle ; \draw  [dash pattern={on 0.84pt off 2.51pt}] (200.73,1227.16) -- (200.73,1328.16)(220.73,1227.16) -- (220.73,1328.16)(240.73,1227.16) -- (240.73,1328.16)(260.73,1227.16) -- (260.73,1328.16)(280.73,1227.16) -- (280.73,1328.16)(300.73,1227.16) -- (300.73,1328.16)(320.73,1227.16) -- (320.73,1328.16)(340.73,1227.16) -- (340.73,1328.16)(360.73,1227.16) -- (360.73,1328.16)(380.73,1227.16) -- (380.73,1328.16) ; \draw  [dash pattern={on 0.84pt off 2.51pt}] (200.73,1227.16) -- (381.73,1227.16)(200.73,1247.16) -- (381.73,1247.16)(200.73,1267.16) -- (381.73,1267.16)(200.73,1287.16) -- (381.73,1287.16)(200.73,1307.16) -- (381.73,1307.16)(200.73,1327.16) -- (381.73,1327.16) ; \draw  [dash pattern={on 0.84pt off 2.51pt}]  ;
\draw    (240.73,1247.48) -- (200.73,1287.48) ;
\draw    (300.73,1307.48) -- (240.73,1247.48) ;
\draw    (320.73,1287.48) -- (300.73,1307.48) ;
\draw    (340.73,1307.48) -- (320.73,1287.48) ;
\draw    (380.73,1267.48) -- (340.73,1308.16) ;
\draw    (200.61,1328.27) -- (200.73,1197.48) ;
\draw [shift={(200.73,1194.48)}, rotate = 90.05] [fill={rgb, 255:red, 0; green, 0; blue, 0 }  ][line width=0.08]  [draw opacity=0] (8.93,-4.29) -- (0,0) -- (8.93,4.29) -- cycle    ;
\draw    (200.39,1287.07) -- (397.58,1287.16) ;
\draw [shift={(400.58,1287.16)}, rotate = 180.03] [fill={rgb, 255:red, 0; green, 0; blue, 0 }  ][line width=0.08]  [draw opacity=0] (8.93,-4.29) -- (0,0) -- (8.93,4.29) -- cycle    ;

\draw (387,1291.29) node [anchor=north west][inner sep=0.75pt]    {$x$};
\draw (208,1188.29) node [anchor=north west][inner sep=0.75pt]    {$y$};
\draw (361.73,1250.56) node [anchor=north west][inner sep=0.75pt]    {$( 9,1)$};

\end{tikzpicture}
\end{center}
\begin{figure}[h]
A path corresponding to a generalized $(5,4,2)-$Dyck sequence
\end{figure}

\begin{center}
\begin{tikzpicture}[x=0.75pt,y=0.75pt,yscale=-1,xscale=1]
\draw  [draw opacity=0][dash pattern={on 0.84pt off 2.51pt}] (200.48,1477.73) -- (381.73,1477.73) -- (381.73,1638.1) -- (200.48,1638.1) -- cycle ; \draw  [dash pattern={on 0.84pt off 2.51pt}] (200.48,1477.73) -- (200.48,1638.1)(220.48,1477.73) -- (220.48,1638.1)(240.48,1477.73) -- (240.48,1638.1)(260.48,1477.73) -- (260.48,1638.1)(280.48,1477.73) -- (280.48,1638.1)(300.48,1477.73) -- (300.48,1638.1)(320.48,1477.73) -- (320.48,1638.1)(340.48,1477.73) -- (340.48,1638.1)(360.48,1477.73) -- (360.48,1638.1)(380.48,1477.73) -- (380.48,1638.1) ; \draw  [dash pattern={on 0.84pt off 2.51pt}] (200.48,1477.73) -- (381.73,1477.73)(200.48,1497.73) -- (381.73,1497.73)(200.48,1517.73) -- (381.73,1517.73)(200.48,1537.73) -- (381.73,1537.73)(200.48,1557.73) -- (381.73,1557.73)(200.48,1577.73) -- (381.73,1577.73)(200.48,1597.73) -- (381.73,1597.73)(200.48,1617.73) -- (381.73,1617.73)(200.48,1637.73) -- (381.73,1637.73) ; \draw  [dash pattern={on 0.84pt off 2.51pt}]  ;
\draw    (200.48,1657.73) -- (200.73,1451.17) ;
\draw [shift={(200.73,1448.17)}, rotate = 90.07] [fill={rgb, 255:red, 0; green, 0; blue, 0 }  ][line width=0.08]  [draw opacity=0] (8.93,-4.29) -- (0,0) -- (8.93,4.29) -- cycle    ;
\draw    (200.39,1537.68) -- (397.58,1537.77) ;
\draw [shift={(400.58,1537.77)}, rotate = 180.03] [fill={rgb, 255:red, 0; green, 0; blue, 0 }  ][line width=0.08]  [draw opacity=0] (8.93,-4.29) -- (0,0) -- (8.93,4.29) -- cycle    ;
\draw    (220.48,1517.73) -- (200.86,1537.51) ;
\draw    (280.48,1577.73) -- (220.48,1517.73) ;
\draw    (320.48,1537.73) -- (280.48,1577.73) ;
\draw    (340.48,1557.73) -- (320.48,1537.73) ;
\draw    (380.48,1517.73) -- (340.48,1557.73) ;
\draw [color={rgb, 255:red, 0; green, 0; blue, 0 }  ,draw opacity=1 ] [dash pattern={on 4.5pt off 4.5pt}]  (320.48,1617.73) -- (280.48,1577.73) ;
\draw [color={rgb, 255:red, 0; green, 0; blue, 0 }  ,draw opacity=1 ] [dash pattern={on 4.5pt off 4.5pt}]  (340.48,1597.73) -- (320.48,1617.73) ;
\draw [color={rgb, 255:red, 0; green, 0; blue, 0 }  ,draw opacity=1 ] [dash pattern={on 4.5pt off 4.5pt}]  (380.48,1637.73) -- (340.48,1597.73) ;



\draw (208,1442.29) node [anchor=north west][inner sep=0.75pt]    {$y$};
\draw (387,1540.29) node [anchor=north west][inner sep=0.75pt]    {$x$};
\draw (361.73,1499.56) node [anchor=north west][inner sep=0.75pt]    {$( 9,1)$};
\draw (351.48,1640.13) node [anchor=north west][inner sep=0.75pt]    {$( 9,-5)$};
\end{tikzpicture}
\end{center}
\begin{figure}[h]
A path corresponding to a sequence with $5$ $0$s and $4$ $1$s  which is not a  generalized $(5,4,2)-$Dyck sequence
\end{figure}
\end{proof}

We shall denote the number $\binom{a+b}{a}-\binom{a+b}{a+c}$ by $\mathrm{GC}(a,b,c)$ and denote by $C_n$ the $n-$th Catalan number.

\begin{lemma}\label{combiidentity}
Let $n>1$ be a positive integer. The following identity holds.
$$1+n^2+\sum_{i=2}^{n-1}\sum_{j=2}^{n-1} ij\mathrm{GC}(n-i-1,n-j-1,i-1)=\frac{1}{n+2}\binom{2n+2}{n+1}=C_{n+1}.$$
\end{lemma}

\begin{proof}
The assertion follows from two elementary combinatorial identities below.

\begin{equation}\label{eq1}
\sum_{j=2}^{n-1}j\left[\binom{2n-i-j-2}{n-j-1}-\binom{2n-i-j-2}{n-2}\right]=\frac{i-1}{n}\binom{2n-i}{n-i+1};
\end{equation}

\begin{equation}\label{eq2}
\sum_{i=2}^{n+1}\frac{i(i-1)}{n}\binom{2n-i}{n-i+1}=C_{n+1}
\end{equation}

The equation \ref{eq1} can be verified using double induction. It is straightforward to verify the identity for the cases where $n=2$ or $i=1$.
Suppose that the identity holds for $n=m-1$ and $i=k-1$ and also holds for $n=m$ and $i=k.$
 Note that
$$\sum_{j=2}^{m-1}j\left[\binom{2m-(k+1)-j-2}{m-j-1}-\binom{2m-(k+1)-j-2}{m-2}\right]$$
$$=\left\{\sum_{j=2}^{m-1}j\left[\binom{2m-k-j-2}{m-j-1}-\binom{2m-k-j-2}{m-2}\right]\right\}$$$$-\left\{\sum_{j=2}^{m-2}j\left[\binom{2(m-1)-(k-1)-j-2}{(m-1)-j-1}-\binom{2(m-1)-(k-1)-j-2}{(m-1)-2}\right]\right\}$$
By induction 
$$=\frac{k-1}{m}\binom{2m-k}{m-k+1}-\frac{k-2}{m-1}\binom{2m-k-1}{m-k+1}$$
$$=\frac{(2m-k)(k-1)\frac{1}{m}}{m-k+1}\binom{2m-k-1}{m-k}-\frac{k-2}{m-k+1}\binom{2m-k-1}{m-k}$$
$$=\frac{k}{m}\binom{2m-k-1}{m-k}$$

The equation \ref{eq2} can be verified using generating function and convolution products.

We start with the following expansion:
$$\frac{(k-1)!x}{(1-x)^{k}}=\sum_{j=1}^{\infty}(j(j+1)\cdots (k-2+j))x^j.$$
Then $$\frac{2x}{n(n+1)(n+2)}\frac{(n+2)!x}{(1-x)^{n+3}}=\frac{2(n-1)!x^2}{(1-x)^{n+3}}=\frac{2x}{(1-x)^3}\frac{(n-1)!x}{(1-x)^{n}}.$$
The expansion of the left hand side shows that the $(n+1)-$th term is
$$\frac{2x}{n(n+1)(n+2)}n(n+1)(n+2)\cdots (2n+1)x^n=2(n+3)(n+4)\cdots (2n+1)x^{n+1}.$$ 
The expansion of the right hand side shows the $(n+1)-$th term is
$$\sum_{j=1}^{n}j(j+1)(2n-j-1)(2n-j-2)\cdots (n-j+1)x^{n+1}.$$ 
Therefore we have the equality
$$\sum_{j=1}^{n}j(j+1)(2n-j-1)(2n-j-2)\cdots (n-j+1)=2(n+3)(n+4)\cdots (2n+1)$$
$$=2\frac{(2n+1)!}{(n+2)!}=\frac{2(n+1)(2n+1)!}{(n+1)(n+2)!}=\frac{(2n+2)!}{(n+2)!(n+1)}$$$$=\frac{n!(2n+2)!}{(n+1)!n!(n+1)(n+2)}=\frac{n!}{n+2}\binom{2n+2}{n+1}=n!C_{n+1}$$
Dividing both sides by $n!$ yields
$$\sum_{j=1}^{n}\frac{j(j+1)(2n-j-1)(2n-j-2)\cdots (n-j+1)}{n!}=C_{n+1}$$
$$\sum_{j=1}^{n}\frac{j(j+1)}{n}\cdot\frac{(2n-j-1)(2n-j-2)\cdots (n-j+1)}{(n-1)!}=C_{n+1}$$
$$\sum_{j=1}^{n}\frac{j(j+1)}{n}\cdot\frac{(2n-j-1)!}{(n-j)!(n-1)!}=C_{n+1}$$
$$\sum_{j=1}^{n}\frac{j(j+1)}{n}\binom{2n-j-1}{n-j}=C_{n+1}$$
$$\sum_{j=2}^{n+1}\frac{j(j-1)}{n}\binom{2n-j}{n-j+1}=C_{n+1}$$
$$1+n^2+\sum_{j=2}^{n-1}\frac{j(j-1)}{n}\binom{2n-j}{n-j+1}=C_{n+1}.$$
 
By \ref{eq2}, we have 
$$1+n^2+\sum_{i=2}^{n-1}\frac{i(i-1)}{n}\binom{2n-i}{n-i+1}=C_{n+1}.$$
Then plug in \ref{eq1}, we have the desired formula.
\end{proof}

An extended Dyck word of length $2n$ is said to have type $(a,b)$ if it is of the form $$\underbrace{00\dots0}_\textrm{a times}1\dots0\underbrace{11\dots 1}_\textrm{b times}.$$

\begin{lemma}\label{numbering}
The number of type $(a,b)$ extended Dyck words of length $2n$ is $\mathrm{GC}(n-a-1, n-b-1, a-1)$.
\end{lemma}

\begin{proof}
Such a word starts with $a$ $0$'s and one $1$ and ends with one $0$ and $b$ $1$'s. In-between these two segments, any
generalized  $(n-a-1, n-b-1, a-1)-$Dyck sequence can be inserted to create a unique type $(a,b)$ extended Dyck word of length $2n$ and any type $(a,b)$ extended Dyck word of length $2n$ can be obtained this way. Then the assertion follows from the above argument and Theorem \ref{counting}.
\end{proof}

A $2n$-trimmed Dyck word $w$ is said to be associated with an extended Dyck word (of length $2n$) $u=\underbrace{00\dots0}_\textrm{a times}1\dots0\underbrace{11\dots 1}_\textrm{b times}$ if it is obtained by substituting the first $a$ $0$s with $a'$ ($a'\leq a$) $0$s and substituting the final $b$ $0$'s with $b'$ ($b'\leq b$) 1's. 

\begin{lemma}\label{association}
Every $2n-$trimmed Dyck word is associated to a unique extended Dyck word of length $2n$.
\end{lemma}

\begin{proof}
Let $w$ be a $2n-$trimmed Dyck word. Therefore $w$ has the form 

(a) $0\dots 01w' 01\dots 1$ or 

(b) $0\dots 01\dots 1$.

In the situation (b) let $w''$ be the empty word and in the situation (a) let $w''$ be $1w'0$. Suppose that $w''$ consists $a$ 0's and $b$ 1s.
We construct a word $$\underbrace{0\dots 0}_\textrm{$n-a$ times}w''\underbrace{1\dots 1}_\textrm{$n-b$ times}.$$
Such a word is clearly an extended Dyck word of length $2n$.
\end{proof}  
  
\begin{theorem}
The number of  $2n-$trimmed Dyck words is $C_{n+1}-1$.
\end{theorem}

\begin{proof}
By Lemma \ref{association}, each $2n-$trimmed Dyck word is associated to a unique extended Dyck word of length $2n$.
Suppose that such an extended Dyck word is of type $(a,b)$. Then evidently the number of $2n-$trimmed Dyck words associated with it is $ab$.
By Lemma \ref{numbering}, there are $\mathrm{GC}(n-a-1, n-b-1, a-1)$ type $(a,b)$ extended Dyck words of length $2n$.
Therefore all type $(a,b)$ extended Dyck words of length $2n$ give rise to $ab\mathrm{GC}(n-a-1, n-b-1, a-1)$ different $2n-$trimmed Dyck words.
Finally the theorem follows from Lemma \ref{combiidentity}.
\end{proof}

\section{Acknowledgment}

The author thanks Professor Matthew Dyer for useful communication which inspires the condensation in Section \ref{condensation}.

\section{Funding Declaration}

None.





\end{document}